\documentclass[a4paper,12pt]{amsart}

\usepackage{amsmath}
\usepackage{amssymb}
\usepackage{amsfonts}

\newcommand{\bA}{\mathbf{A}}

\newtheorem{thm}{Theorem}[section]
\newtheorem{prop}[thm]{Proposition}
\newtheorem{lem}[thm]{Lemma}

\newtheorem{rem}[thm]{Remark}
\newtheorem{rems}[thm]{Remarks}
\newtheorem{defi}[thm]{Definition}
\newtheorem{exo}{\bf\large Exercice}

\newcommand{\R}{\mathbb{R}}
\newcommand{\Z}{\mathbb{Z}}
\newcommand{\N}{\mathbb{N}}
\newcommand{\C}{\mathbb{C}}

\newcommand{\I}{\infty}

\newcommand{\dd}{\mathbf{D}}

\newcommand{\Sum}{\displaystyle \sum}

\newcommand{\Int}{\displaystyle \int}

\newcommand{\Inf}{\displaystyle \inf}

\newcommand{\Lim}{\displaystyle \lim}
\newcommand{\Liminf}{\displaystyle \liminf}
\newcommand{\Limsup}{\displaystyle \limsup}
\newcommand{\Max}{\displaystyle \max}

\newcommand{\beq}{\begin{eqnarray}}
\newcommand{\eeq}{\end{eqnarray}}
\newcommand{\bq}{\begin{equation}}
\newcommand{\eq}{\end{equation}}
\newcommand{\beqn}{\begin{eqnarray*}}
\newcommand{\eeqn}{\end{eqnarray*}}
\newcommand{\bex}{\begin{exo}}
\newcommand{\eex}{\end{exo}}
\newcommand{\ben}{\begin{enumerate}}
\newcommand{\een}{\end{enumerate}}

\let\al=\alpha

\let\eps=\varepsilon

\let\lam=\lambda

\let\wt=\widetilde
\let\wh=\widehat

\def\cB{{\mathcal B}}
\def\cC{{\mathcal C}}
\def\cD{{\mathcal D}}

\def\cL{{\mathcal L}}

\def\cP{{\mathcal P}}

\def\eqdefa{\buildrel\hbox{\footnotesize def}\over =}

\author{Hajer Bahouri}
\address{LAMA UMR CNRS 8050,
Universit\'e Paris-Est Cr\'eteil \\ 61, avenue du G\'en\'eral de Gaulle\\
94010 Cr\'eteil Cedex\\ France.} \email{hbahouri@math.cnrs.fr}

\author{Mohamed Majdoub}
\address{University of Tunis ElManar, Faculty of Sciences of Tunis,
Department of Mathematics.} \email{mohamed.majdoub@fst.rnu.tn}

\author{Nader Masmoudi}
\address{New York University \\
The Courant Institute for Mathematical Sciences.}
\email{masmoudi@courant.nyu.edu}
\thanks{N. M is partially supported by an NSF Grant DMS-0703145}

\title[Lack of compactness...]
{Lack of compactness in the 2D critical Sobolev embedding, the general case}
\date{\today}

\begin{document}

\begin{abstract}
This paper is devoted to the description of the lack of
compactness of the Sobolev embedding of $H^1(\R^2)$ in the critical Orlicz space ${\cL}(\R^2)$. It turns out that up to cores our result
is expressed in terms of the concentration-type examples derived
by J.~Moser in \cite{M} as in the radial setting investigated in \cite{BMM}. However, the analysis we used in this work is strikingly different from the one conducted in the radial case which is  based on an $L^ \infty$ estimate far away from the origin and which is no longer valid in the general framework. Within the general framework of $H^1(\R^2)$, the strategy we adopted to build the profile decomposition in terms of examples by Moser concentrated around cores is based on capacity arguments and relies on an extraction process of mass concentrations. The essential ingredient to extract cores consists
 in proving  by contradiction that if the mass responsible for the lack of compactness of the Sobolev embedding in  the Orlicz space is scattered, then the energy used would
  exceed  that of the starting sequence.
\end{abstract}

\subjclass[2000]{...}


\keywords{Sobolev critical exponent, Trudinger-Moser inequality,
Orlicz space, lack of compactness, capacity, Schwarz symmetrization.}

\maketitle

\tableofcontents


\section{Introduction}


\subsection{Critical $2D$ Sobolev embedding}


It is well known that $H^1(\R^2)$ is continuously embedded  in all
Lebesgue spaces $L^p(\R^2)$ for $2\leq p<\infty$, but not in
$L^{\infty}(\R^2 )$.   On the other hand, it
is also known (see for instance  \cite{BCD, Kozono}) that  $H^1(\R^2)$ embeds
in  ${\rm BMO} (\R^2) \cap L^2 (\R^2)    $,  where ~$BMO(\R^d )$
denotes the space of bounded mean oscillations which is  the space
of locally integrable functions~$f$ such that \[ \|f\|_{BMO}\eqdefa
\sup_{B}\frac 1 {|B|}\int_B |f-f_B|\,dx< \infty \quad
\mbox{with}\quad f_B\eqdefa \frac 1 {|B|}\int_B f \,dx.\] The above
supremum being taken over
the set of Euclidean balls~$B$,~$|\cdot |$ denoting the Lebesgue measure.

For the sake of geometric problems and the understanding of   features of solutions
to nonlinear partial differential equations with
exponential growth, we investigate in this paper the lack of compactness of Sobolev embedding of $H^1(\R^2)$ in the Orlicz space ${\mathcal L}(\R^2)$ (defined below) which is  not
  comparable to~$BMO(\R^2 )$ (for a proof of this fact, one can consult \cite{BMM}). Notice
 that the lack of compactness of
$$\dot H^1(\R^d)\hookrightarrow BMO(\R^d)$$
was characterized in \cite{BCG} using a wavelet-based profile decomposition.

Note that  in higher dimension, are available several works  that highlight the role of the study of the lack of compactness in critical Sobolev embedding to  the description of bounded energy sequences of solutions of nonlinear partial differential equations. Among others, one can mention  \cite{BG, BG2, IB, km, ker,La, Ma, Tao}.\\

For the convenience of the reader, let us introduce the so-called Orlicz spaces on $\R^d$ and some
related basic facts (for a
complete presentation and more details, we refer the reader to
\cite{Orlicz-Book}).
\begin{defi}\label{deforl}\quad\\
Let $\phi : \R^+\to\R^+$ be a convex increasing function such that
$$
\phi(0)=0=\lim_{s\to 0^+}\,\phi(s),\quad
\lim_{s\to\infty}\,\phi(s)=\infty.
$$
We say that a measurable function $u : \R^d\to\C$ belongs to
$L^\phi$ if there exists $\lambda>0$ such that
$$
\Int_{\R^d}\,\phi\left(\frac{|u(x)|}{\lambda}\right)\,dx<\infty.
$$
We denote then \bq \label{norm}
\|u\|_{L^\phi}=\Inf\,\left\{\,\lambda>0,\quad\Int_{\R^d}\,\phi\left(\frac{|u(x)|}{\lambda}\right)\,dx\leq
1\,\right\}. \eq
\end{defi}
\begin{rems}\quad\\
\noindent$\bullet$
It is easy to check that $\|\cdot\|_{L^\phi}$ is a norm  on the  $\C$-vector space $L^\phi$
which is invariant under translations and oscillations. In particular, we have for any $ u \in L^\phi$
\begin{equation}\label{moeql}\|u\|_{L^\phi} = \|\,|u|\,\|_{L^\phi} \,.\end{equation} \noindent$\bullet$ For $\phi(s)=s^p,\, 1\leq
p<\infty$,  $L^\phi$ is nothing else than the  Lebesgue space
$L^p$. \\ \noindent$\bullet$ One can  easily  verify that the number $1$  in
\eqref{norm} may be replaced by any positive constant and that this changes
the norm $\|\cdot\|_{L^\phi}$ by an equivalent norm.\end{rems}

In what follows we shall fix  $d=2$, $\phi(s)={\rm e}^{ s^2}-1$ and
denote the Orlicz space $L^\phi$ by ${\mathcal L}$ endowed with the
norm $\|\cdot\|_{\mathcal L}$ where the number $1$ is replaced by
the constant $\kappa$ that will be fixed in identity  \eqref{Mos2}
below.  As it is already mentioned, this does not have any
impact on  the definition of the Orlicz space.

The 2D critical Sobolev embedding in the Orlicz space ${\mathcal L}$
states as follows:
\begin{prop}\label{prop2D-embed}
\begin{equation}
\label{2D-embed} \|u\|_{{\mathcal
L}}\leq\frac{1}{\sqrt{4\pi}}\|u\|_{H^1}.
\end{equation}
\end{prop}

Let us point out
that the embedding~\eqref{2D-embed} derives immediately from the
following Trudinger-Moser  inequality proved in  \cite{Ruf}:
\begin{prop}\label{proptm}
\begin{equation}
\label{Mos2} \sup_{\|u\|_{H^1}\leq 1}\;\;\Int_{\R^2}\,\left({\rm
e}^{4\pi |u|^2}-1\right)\,dx:=\kappa<\infty,
\end{equation} and this is false for $\alpha>4\pi$.
\end{prop}
If we only  require that  $\|\nabla u\|_{L^2(\R^2)}\leq1$
rather than $\|u\|_{H^1(\R^2)}\leq1$,   then the following sharp
 Trudinger-Moser type inequality holds:
\begin{prop} \label{mostrud}
 A constant $C$ exists such that
\begin{equation}
\label{Mos1} \Int_{\R^2}\,\left(\frac{ {\rm e}^{4\pi
|u|^2}-1 }{1 + |u|^2 }\right)\,dx\leq C \|u\|_{L^2(\R^2)}^2,
\end{equation}
for all $u$ in $H^1(\R^2)$ such that $\|\nabla
u\|_{L^2(\R^2)}\leq1$.
\end{prop}
\begin{rems}\quad\\
\noindent$\bullet$
These Trudinger-Moser inequalities established in \cite{IMN11TM}
  will be of constant use along this paper. Note that  a sharp form of  Trudinger-Moser  inequalities  in bounded domain was  obtained in \cite{AD} and a subtle improvement of these inequalities was  demonstrated  in \cite{Struwe10}. For further results on the subject, see
\cite{AT, A, CC,Flu,FLS,IMM07,M,Tru} and the references therein.\\
\noindent$\bullet$ The injection of $H^1(\R^2)$ in  ${\mathcal L}(\R^2)$  is sharp
within the context of Orlicz spaces.  However, this embedding can be improved if we allow
different function spaces than Orlicz spaces. More precisely \bq
\label{sharp-embed} H^1(\R^2)\hookrightarrow BW(\R^2), \eq where the
Brezis-Wainger space $BW(\R^2)$ is defined via
$$
\|u\|_{BW}:=\Big(\int_0^1\Big(\frac{u^{\sharp}(t)}{\log({\rm
e}/t)}\Big)^2\,\frac{dt}{t}\Big)^{1/2}+\Big(\int_1^\I\,u^{\sharp}(t)^2\,dt\Big)^{1/2},
$$
where $u^{\sharp}$ denotes the symmetric decreasing rearrangement of $u$. \\

The embedding \eqref{sharp-embed} is sharper than \eqref{2D-embed} as $BW(\R^2)\varsubsetneqq \cL(\R^2)$. It is also optimal with respect to all rearrangement invariant Banach function spaces. For more details on this subject, we refer the reader to \cite{BW, Cianchi, EKP, Hans, HMT, MP-PAMS}. \\
\noindent$\bullet$
To end this short introduction to Orlicz spaces, let us notice  that the Orlicz space ${\cL}$ behaves like $L^2$ for functions in $H^1\cap L^\infty$ (see \cite{BMM} for more details).\end{rems}


\subsection{Development on  lack of
compactness of Sobolev embedding  in the Orlicz space}


The embedding~$H^1\hookrightarrow{\mathcal L}$  is not compact at
least for two reasons. The first reason is the lack of compactness
at infinity. A typical example is given by $u_k(x)=\varphi(x+x_k)$ where
$0\neq\varphi\in{\mathcal D}$ and $|x_k|\to\infty$. The second
reason is of concentration-type derived by  P.-L.  Lions
\cite{Lions1, Lions2} and illustrated by the following  fundamental
example  $f_{\alpha}$ defined  for $\alpha >0 $  by:
\begin{eqnarray*}
 f_{\alpha}(x)&=&\; \left\{
\begin{array}{cllll}0 \quad&\mbox{if}&\quad
|x|\geq 1,\\\\ -\frac{\log|x|}{\sqrt{2\alpha\pi}} \quad
&\mbox{if}&\quad {\rm e}^{-\alpha}\leq |x|\leq 1 ,\\\\
\sqrt{\frac{\alpha}{2\pi}}\quad&\mbox{if}&\quad |x|\leq {\rm
e}^{-\alpha}.
\end{array}
\right.
\end{eqnarray*}

Indeed, it can be seen easily that
$f_{\alpha}\rightharpoonup 0$ in $H^1(\R^2)$ as $\alpha$ tends to either $0$ or $\infty$. However, the lack of compactness of this sequence
in  the Orlicz space ${\cL}$ occurs only when $\alpha$ goes to
infinity and we have by straightforward  computations  (detailed for instance  in \cite{BMM}):
\bq
\label{behav}\|f_{\alpha}\|_{\mathcal L}\rightarrow \frac{1}{\sqrt{4\pi}}\quad\mbox{as}\quad \alpha\to\infty\quad\mbox{and}\quad \|f_{\alpha}\|_{\mathcal L}\rightarrow
0\quad\mbox{as}\quad \alpha\to0,\eq
and that the essential contribution of $\|f_{\alpha}\|_{\mathcal L}$ when $\alpha \to \infty$ comes from the  balls $B(0,{\rm
e}^{-\alpha})$.\\

The difference between the behavior of these families in the Orlicz
space when $\alpha\to 0$ or $\alpha\to\infty$ comes from the fact
that the concentration effect is only displayed by this family when
$\alpha\to\infty$.\\

In fact, in \cite{Lions1, Lions2}, P. -L. Lions highlighted the fact that the defect of compactness of Sobolev embedding  in the Orlicz space ${\cL}$, under compactness at infinity,  is due to a concentration phenomena. More precisely, he proved  the following result.
\begin{prop}
\label{macroscop}
Let $(u_n)$ be a sequence in
$H^1(\R^2)$ such that
$$
u_n\rightharpoonup 0\;\;\mbox{in}\;\;
H^1,\;\;\limsup_{n\to\infty}\,\|u_n\|_{\cL}:=A>0\;\;\mbox{and}\;\;\lim_{R\to\infty}
\limsup_{n\to\infty}\|u_n\|_{L^2(|x|>R)}=0\,.
$$
Then, there exists a point $x_0\in\R^2$ and a constant $c>0$ such that
\bq
\label{concent}
|\nabla u_n(x)|^2\,dx\rightharpoondown\,\mu\geq c\,\delta_{x_0},\, \mbox{as} \,\, n\to\infty,
\eq
weakly in the sense of measures.
\end{prop}
\begin{rems}\quad\\
\noindent$\bullet$
The hypothesis of compactness at infinity
\begin{equation}\label{compl2}
\lim_{R\to\infty}
\limsup_{n\to\infty}\|u_n\|_{L^2(|x|>R)}=0
\end{equation}
is necessary to get  \eqref{concent}. For instance,
 $u_n(x)=\frac{1}{n}\,{\rm e}^{-|\frac{x}{n}|^2}$   satisfies $\|u_n\|_{L^2}=\sqrt{\frac{\pi}{2}}$,
 $\|\nabla u_n\|_{L^2}\to 0$ and  $\|u_n\|_{\cL}\geq \sqrt{\frac{\pi}{2\kappa}} $. Indeed,
 $$
 \|u_n\|_{L^2}^2=\int_{\R^2}\,{\rm e}^{-2|y|^2}\,dy=2\pi\int_0^\infty\,r {\rm e}^{-2 r^2}\,dr=\frac{\pi}{2}\,
 $$
and the last assertion follows from the more general estimate
\bq
\label{leborlicz}
\frac{1}{\sqrt{\kappa}}\,\|u\|_{L^2}\leq \|u\|_{\cL}\,.
\eq
In fact by virtue of Rellich's theorem, the compactness at infinity assumption \eqref{compl2} coupled with the boundedness in $H^1$
implies the strong convergence of the sequence $(u_n)$  to zero in $L^2$.\\
\noindent$\bullet$
Let us notice that in the case where
\begin{equation}\label{conpoint}|\nabla u_n(x)|^2\,dx\rightharpoondown\, c\,\delta_{x_0},\, \mbox{as} \,\, n\to\infty\end{equation} weakly in the sense of measures, we have compactness in the Orlicz space away from the point $x_0$. More precisely, we have the following lemma:
\begin{lem}
\label{genradaway} Let $(u_n)$ be a bounded sequence in
$H^1(\R^2)$ weakly convergent to $0$, satisfying the hypothesis of compactness at infinity \eqref{compl2} and the hypothesis of concentration of the total mass \eqref{conpoint}.
Then for any nonnegative reals ~$M$ and ~$\alpha$, we have
\begin{equation}
\label{nonraway} \int_{|x - x_0| >  M } \left({\rm
e}^{|\alpha \,u_n(x)|^2}-1\right)\,dx \to 0, \quad n \to \infty.
\end{equation}
\end{lem}
\begin{proof}
For fixed ~$M$,  let us  consider ~$\varphi \in {\cC}^\infty(\R^2)$ satisfying ~$0 \leq \varphi \leq 1 $ and such that
\begin{eqnarray*}
  \left\{
\begin{array}{cllll}\varphi(x)= \,  0   \quad&\mbox{if}&\quad
|x-x_0|\leq \frac M 2  \quad \mbox{and}\\\\ \varphi(x)= \,  1  \quad&\mbox{if}&\quad
|x-x_0|\geq M.
\end{array}
\right.
\end{eqnarray*}
It is obvious that
 $$ \|\varphi\,u_n\|_{H^1} \stackrel{n\to\infty}\longrightarrow 0,$$
which implies for $n$ big enough thanks to Trudinger-Moser estimate (\ref{Mos1})
$$ \int_{\R^2} \left({\rm
e}^{|\alpha \,\varphi(x) u_n(x)|^2}-1\right)\,dx \lesssim \|\varphi\,u_n\|_{L^2}^2 \to 0, \quad \mbox{as} \quad n \to \infty$$
and ends the proof of the lemma.
\end{proof}

\noindent$\bullet$ In fact in \cite{Lions1, Lions2}, P. -L.  Lions relies the defect of compactness of Sobolev embedding in the Orlicz space as well as  in Lebesgue spaces to concentration phenomena.  Although these  results are  expressed in the same manner by means of defect measures, they are actually of different nature. Indeed, we shall prove  in this article that the lack of compactness of Sobolev embedding in the Orlicz space
can be described in terms of an orthogonal asymptotic decomposition whose
elements  are    completely
different  from the ones involving in the  decomposition derived by P. G\'erard in \cite{Ge2} in the framework of Sobolev embedding  in Lebesgue spaces or by S. Jaffard in \cite{jaffard} in the  more general case of  Riesz potential   spaces. \\
\noindent$\bullet$ Let us mention that the lack of compactness  was also studied
in \cite{BC} for a bounded sequence in $H_0^{1}(D, \R^3)$ of solutions of an elliptic problem, with $D$
the open unit disk of $\R^2$ and in \cite{St} and \cite{So} for the critical injections
of $W^{1,2}(\Omega)$ in Lebesgue space
and of $W^{1,p}(\Omega)$ in Lorentz spaces
respectively, with $\Omega$ a bounded domain of $\R^d$.
Similarly, the issue was addressed in \cite{ST} in an abstract Hilbert space
framework and in \cite{Benameur} in  the Heisenberg group context.\\
\noindent$\bullet$
Recently  in \cite{BCG}, the
wavelet-based profile decomposition introduced by S. Jaffard in \cite{jaffard} was revisited in order to treat a larger range of examples of critical embedding of functions spaces $ X \hookrightarrow Y $
including Sobolev, Besov, Triebel-Lizorkin, Lorentz, H\"older and BMO spaces and for that purpose, two generic properties on the spaces $X$ and $Y$ was  identified to  build the profile decomposition in a unified way. (One can consult ~\cite{BCD} and the references therein for an introduction to spaces listed above).
 \end{rems}
In line with  the results of P. -L.  Lions \cite{Lions1, Lions2},  we investigated  in \cite{BMM} the lack of
compactness of Sobolev embedding of $H^1_{rad}(\R^2)$ in the Orlicz
space ${\mathcal L}(\R^2)$ following the approach of P. G\'erard in
\cite{Ge2} and S. Jaffard in \cite{jaffard} which consists in
characterizing the lack of compactness of the critical Sobolev embedding in Lebesque spaces in  terms   of orthogonal profiles.
In order to recall our  result in a clear
way, let us first  introduce some objects.
\begin{defi}
\label{ortho} We shall denote by  scale any sequence $\underline{\alpha}:=(\alpha_n)$
of positive real numbers going to infinity and we  shall say that two
scales ${\underline\alpha}$ and ${\underline\beta}$ are orthogonal (in short ${\underline\alpha} \perp {\underline\beta}$) if
    $$
   \Big|\log\left({\beta_n}/{\alpha_n}\right)\Big|\to\infty.
    $$
We shall designate by profile any function $\psi$ belonging to the set
$$
{\cP}:=\Big\{\;\psi\in L^2(\R,{\rm e}^{-2s}ds);\;\;\; \psi'\in
L^2(\R),\;\psi_{|]-\infty,0]}=0\,\Big\}.
$$
\end{defi}
\begin{rems}\label{remsrad2}\quad\\
\noindent$\bullet$ The limitation for scales tending to infinity  is justified by the behavior of $\|f_{\alpha}\|_{\mathcal L}$ stated in \eqref{behav}.\\
\noindent$\bullet$ The set ${\cP}$ is invariant under positive translations. More precisely, if $\psi\in{\cP}$ and $a\geq 0$ then $\psi_a(s):=\psi(s -a)$ belongs to ${\cP}$. \\
\noindent$\bullet$
Let us remark that each profile belongs to the H\"older space $C^{\frac 1 2} (\R)$. In particular, any $\psi \in {\cP}$ is continuous.
\\
\noindent$\bullet$ It will be useful to notice that
\begin{equation}\label{behavpsi}
\frac{\psi(s)}{\sqrt{s}} \to 0 \quad\mbox{as}\quad s \to 0 \quad\mbox{and}\quad\mbox{as}\quad s \to \infty.
\end{equation}
Indeed, since $\psi' \in L^2$ and $$ \Big|\psi(s)\Big| = \Big|\int_0^{s}\,\psi'(\tau)\,d\tau\Big| \leq \sqrt{s}\left(\int_0^{s}\,\psi'^2(\tau)\,d\tau\right)^{1/2},$$
we  get the first part of the assertion \eqref{behavpsi}. Similarly, for  $s > A$ we have
$$ \Big|\frac{\psi(s)}{\sqrt{s}}\Big| \leq \Big|\frac{\psi(s)- \psi(A)}{\sqrt{s}}\Big|+ \Big|\frac{\psi(A)}{\sqrt{s}}\Big|\leq \frac{\sqrt{s- A}}{\sqrt{s}} \left(\int_A^{s}\,\psi'^2(\tau)\,d\tau\right)^{1/2}+ \Big|\frac{\psi(A)}{\sqrt{s}}\Big|, $$which easily ensures that $\frac{\psi(s)}{\sqrt{s}} \to 0$ as $s$ tends to $\infty$. Now $\psi$ being a continuous function, we deduce that the $\displaystyle\sup_{s>0}\,\frac{|\psi(s)|}{\sqrt{s}}$ is achieved on $]0, +\infty[$. \\
\noindent$\bullet$ Finally, let us observe that for a scale $\underline{\alpha}$ and a profile $\psi$, if we denote by
$$ g_{\underline{\alpha}, \psi} (x):=\sqrt{\frac{\alpha_n}{2\pi}}\;\psi\Big(\frac{-\log|x|}{\alpha_n}\Big)\,,$$
then for any $\lambda>0$,
\begin{equation}\label{inara}g_{\underline{\alpha}, \psi}=g_{\lambda\underline{\alpha}, \psi_{\lambda}},\end{equation}
where $ \psi_{\lambda}(t)=\frac{1}{\sqrt{\lambda}}\,\psi(\lambda
t)$.
\end{rems}

In \cite{BMM}, we established that
the lack of compactness of the embedding
$$
H^1_{rad}\hookrightarrow{\mathcal L},
$$
can be reduced to generalization of  the  example by Moser. More precisely, we  proved
that the lack of compactness of this embedding can be described in
terms of an asymptotic decomposition  as follows:
\begin{thm}
\label{main} Let $(u_n)$ be a bounded sequence in
$H^1_{rad}(\R^2)$ such that \bq \label{main-assum1}
u_n\rightharpoonup 0, \eq \bq \label{main-assum2}
\limsup_{n\to\infty}\|u_n\|_{\mathcal L}=A_0>0, \quad \quad
\mbox{and} \eq \bq \label{main-assum3} \lim_{R\to\infty}
\limsup_{n\to\infty}\|u_n\|_{L^2(|x|>R)}=0. \eq Then, there
exists a sequence $(\underline{\alpha}^{(j)})$ of pairwise
orthogonal scales and a sequence of profiles $(\psi^{(j)})$ in
${\cP}$ such that, up to a subsequence extraction, we have for all
$\ell\geq 1$, \bq \label{decomp}
u_n(x)=\Sum_{j=1}^{\ell}\,\sqrt{\frac{\alpha_n^{(j)}}{2\pi}}\;\psi^{(j)}\left(\frac{-\log|x|}{\alpha_n^{(j)}}\right)+{\rm
r}_n^{(\ell)}(x),\quad\limsup_{n\to\infty}\;\|{\rm
r}_n^{(\ell)}\|_{\mathcal
L}\stackrel{\ell\to\infty}\longrightarrow 0. \eq Moreover, we have
the following orthogonality equality \bq \label{ortogonal} \|\nabla
u_n\|_{L^2}^2=\Sum_{j=1}^{\ell}\,\|{\psi^{(j)}}'\|_{L^2}^2+\|\nabla
{\rm r}_n^{(\ell)}\|_{L^2}^2+\circ(1),\quad n\to\infty. \eq
\end{thm}
\begin{rems}\label{remsradial}\quad\\
\noindent$\bullet$ Note that the  example by Moser can be written under the form
$$
f_{\alpha_n}(x)=\sqrt{\frac{\alpha_n}{2\pi}}\;{\mathbf L}\,\left(\frac{-\log|x|}{\alpha_n}\right),
$$
where
\begin{eqnarray*}
{\mathbf L}(s)&=&\; \left\{
\begin{array}{cllll}0 \quad&\mbox{if}&\quad
s\leq 0,\\ s \quad
&\mbox{if}&\quad 0\leq s\leq 1 ,\\
1 \quad&\mbox{if}&s\geq 1.
\end{array}
\right.
\end{eqnarray*}
Let us also point out that $f_{\alpha_n}$ is the minimum energy function  which is equal to the value $\sqrt{\frac{\alpha_n}{2\pi}}$ on the ball $B(0, {\rm e}^{- \alpha_n})$ and which vanishes outside the unit ball (see Lemma \ref{minEner} for more details).\\
\noindent$\bullet$ The approach that we adopted to prove that  result uses in a crucial way the radial setting and particularly the fact we deal with bounded functions far away from the origin thanks to the well known radial estimate (see for instance \cite{BMM} for a sketch of proof).
\begin{equation}
\label{radest}|u(r)|\leq
\frac{C}{\sqrt{r}}\,\|u\|_{L^2}^{\frac{1}{2}}\|\nabla
u\|_{L^2}^{\frac{1}{2}}.\end{equation}
\noindent$\bullet$ It should be emphasized that, contrary to the case of
Sobolev embedding in Lebesgue spaces, where the asymptotic
decomposition derived by P. G\'erard in \cite{Ge2} leads to $$
\|u_n\|^p_{L^p} \to \sum_{j\geq 1} \|\psi^{(j)}\|^p_{L^p},$$ Theorem
\ref{main}  yields that
\begin{equation}
\label{OrliczMax} \|u_n\|_{\cL}\to \sup_{j\geq
1}\,\left(\lim_{n\to\infty}\,\|g_n^{(j)}\|_{\cL}\right),
\end{equation}
where~$g_n^{(j)}(x) = \sqrt{\frac{\alpha_n^{(j)}}{2\pi}}\;\psi^{(j)}\left(\frac{-\log|x|}{\alpha_n^{(j)}}\right)$. \\
Let us also notice that
\bq \label{profile}
\Lim_{n\to\infty}\,\|g_n^{(j)}\|_{{\cL}}=\frac{1}{\sqrt{4\pi}}\,\max_{s>0}\;\frac{|\psi^{(j)}(s)|}{\sqrt{s}}. \eq
For a detailed proof of \eqref{OrliczMax}  and \eqref{profile}, see Propositions 1.15  and 1.18 in \cite{BMM}. A consequence of \eqref{OrliczMax} is that the first profile in Decomposition \eqref{decomp} can be chosen such that up to extraction
\bq \label{profile1} \limsup_{n\to\infty}\|u_n\|_{\mathcal L}=A_0 = \lim_{n\to\infty}\left\|  \sqrt{\frac{\alpha_n^{(1)}}{2\pi}}\;\psi^{(1)}\left(\frac{-\log|x|}{\alpha_n^{(1)}}\right)  \right\|_{\mathcal L}.\eq
Otherwise, under the  orthogonality assumption of the scales and identity \eqref{inara}, we can suppose without loss of generality that for any $j$
$$ \Lim_{n\to\infty}\,\|g_n^{(j)}\|_{{\cL}}=\frac{1}{\sqrt{4\pi}}\,\max_{s>0}\;\frac{|\psi^{(j)}(s)|}{\sqrt{s}} = \frac{1}{\sqrt{4\pi}}\, |\psi^{(j)}(1)|.$$
\noindent$\bullet$ As a by product of the above remark, it may be noted that while the energy of an elementary concentration $g_n^{(j)}$ remains invariant under  translation of  the profiles~$\psi^{(j)}$, it is not the same for the Orlicz norm. For instance, if we consider the elementary concentration
$$ g_{\alpha_n}(x) = \sqrt{\frac{\alpha_n}{2\pi}}\;{\mathbf L}_{a}\,\left(\frac{-\log|x|}{\alpha_n}\right),  \quad a \geq 0$$ inferred from  example by Moser $f_{\alpha_n}$ by translating the profile, it comes $$\|g_{\alpha_n}\|_{\mathcal L}\rightarrow \frac{1}{\sqrt{4\pi (a+1)}},$$
while $\|f_{\alpha_n}\|_{\mathcal L}\rightarrow \frac{1}{\sqrt{4\pi }}$ as $ n $ tends to infinity.\\
\noindent$\bullet$ Let us observe that each elementary concentration $g_n^{(j)}$ is supported in the unit ball. This is due to the fact that in the radial case,  any bounded sequence in $H^1( \R^2)$ is compact away from the origin in  the Orlicz space (see Lemma \ref{genradaway}).\\
\noindent$\bullet$ Taking advantage of Proposition \ref{prop2D-embed},  relations \eqref{decomp} and \eqref{ortogonal}, we deduce that in the case where
\bq \label{concexample}\|u_n\|_{{\mathcal
L}}\sim \frac{1}{\sqrt{4\pi}}\|u_n\|_{H^1},\eq where the symbol
$\sim$ means that the difference goes to zero   as $n$ tends to infinity, we have necessary
$$u_n(x)=\sqrt{\frac{\alpha_n}{2\pi}}\;\psi\left(\frac{-\log|x|}{\alpha_n}\right)+{\rm
r}_n(x),\quad \|{\rm
r}_n\|_{H^1} \to 0 \quad \mbox{as} \quad n\to\infty, $$ with \begin{equation} \label{formprofile}\psi(s)=\frac{1}{\sqrt{s_0}}\,{\mathbf L}(\frac s {s_0}), \quad \mbox{for some} \quad s_0 > 0.\end{equation}
Indeed, by virtue of  \eqref{profile}, equivalence \eqref{concexample}  leads to  $\|u_n\|_{{\mathcal
L}}\sim \frac{1}{\sqrt{4\pi}}\, \displaystyle\max_{s>0}\,\frac{|\psi(s)|}{\sqrt{s}}\,$. Therefore thanks to \eqref{2D-embed}, we get
$$ \|\psi'\|_{L^2} = \max_{s>0}\,\frac{|\psi(s)|}{\sqrt{s}}.$$
Thus, there exists $s_0>0$ such that
$$
\|\psi'\|_{L^2}=\frac{|\psi(s_0)|}{\sqrt{s_0}}\,.
$$
Since $\frac{|\psi(s_0)|}{\sqrt{s_0}} \leq
\left(\int_0^{s_0}\,|\psi'(t)|^2\,dt\right)^{1/2}\,$, we infer that
 $\psi'=0$
on $[s_0,+\infty[$ and then by continuity
$\psi(s)=\sqrt{s_0}$ for any $s\geq s_0$. Finally, the equality case of the Cauchy-Schwarz's inequality
$$ \Big|\psi(s_0)\Big|= \Big|\int_0^{s_0}\,\psi'(\tau)\,d\tau\Big|=\sqrt{s_0}\left(\int_0^{s_0}\,\psi'^2(\tau)\,d\tau\right)^{1/2},
$$
leads to $\psi(s)=\frac{s}{\sqrt{s_0}}$ for  $s\leq s_0$ which ends the proof of claim \eqref{formprofile}. \\
\noindent$\bullet$  Condition \eqref{concexample}  can be understood  as a condition of  concentration of the whole mass and plays a crucial role in the qualitative study of solutions of nonlinear wave equations with exponential growth (see \cite{BMM} for more details).  \\
\noindent$\bullet$ Compared with the decomposition in \cite{Ge2}, we see that there is no cores in \eqref{decomp}. This is justified by the radial setting. \\
\noindent$\bullet$ Note that
up to changing the remainders, each profile may be assumed to be in~${\mathcal D}(]0,+\infty[)$. This is due to the following lemma that will be  useful in the sequel:
\begin{lem}
\label{stabicompact}
Let $\psi\in{\cP}$ a profile and $\varepsilon>0$. Then, there exists $\psi_\varepsilon\in{\cD}(]0,+\infty[)$ such that
$$
\|\psi'-\psi_\varepsilon'\|_{L^2}\leq \varepsilon\,.
$$
\end{lem}
\begin{proof}
Fix $\psi\in{\cP}$ and $\varepsilon>0$. Since $\psi'\in L^2$ and ${\cD}$ is dense in $L^2$, there exists $\chi\in{\cD}$ such that $\|\psi'-\chi\|_{L^2}\leq\frac{\varepsilon}{2}$. Let $\theta_0\in{\cD}$ satisfying $\int_{\R}\,\theta_0(s)\,ds=1$ and set
$$
\tilde{\chi}=\chi-\theta^\lambda\,\int_{\R}\,\chi(s)\,ds,
$$
where $\theta^\lambda(s)=\lambda\,\theta_0(\lambda\,s)$ and $\lambda$ is a positive constant to be chosen later. Clearly $\tilde\chi\in{\cD}$ and $\int_{\R}\,\tilde\chi(s)\,ds=0$. Hence, there exists a smooth compactly supported function whose derivative is $\tilde\chi$. Besides, straightforward computation leads to
$$
\|\psi'-\tilde\chi\|_{L^2}
\leq\frac{\varepsilon}{2}+\|\chi\|_{L^1}\,\sqrt{\lambda}\,\|\theta_0\|_{L^2}\,,
$$
which  concludes the proof by choosing $
\lambda=\frac{\varepsilon^2}{4\|\chi\|_{L^1}^2\,\|\theta_0\|_{L^2}^2}\,$.
\end{proof}
\noindent$\bullet$  Finally, let us point out that
\begin{equation}
\label{Orliczdirac}
|\nabla g_n^{(j)}(x)|^2\,dx\to \|{\psi^{(j)}}'\|_{L^2}^2\, \delta_0,\quad\mbox{in}\quad {\cD}'(\R^2).
\end{equation}
Indeed, straightforward  computations give for any  smooth compactly supported function $\varphi$
\beqn
\Int\,|\nabla g_n^{(j)}(x)|^2\varphi(x)\,dx&=&\frac{1}{2\pi\alpha_n^{(j)}}\,\Int_{0}^1\,\Int_0^{2\pi}\,\Big|{\psi^{(j)}}' \left(\frac{-\log r}{\alpha_n^{(j)}}\right)\Big|^2\,\frac{\varphi(r\cos\theta,r\sin\theta)}{r}\,dr\,d\theta\\
&=&\varphi(0)  \|{\psi^{(j)}}'\|_{L^2}^2 + I_n^{(j)},
\eeqn
where
$$ I^{(j)}_n = \frac{1}{2\pi\alpha_n^{(j)}}\,\Int_{0}^1\,\Int_0^{2\pi} \,\Big|{\psi^{(j)}}' \left(\frac{-\log r}{\alpha_n^{(j)}}\right)\Big|^2\,\frac{\varphi(r\cos\theta,r\sin\theta)-\varphi(0)}{r}\,dr\,d\theta.$$
Since
$\Big|\frac{\varphi(r\cos\theta,r\sin\theta)-\varphi(0)}{r}\Big|\leq
\|\nabla\varphi\|_{L^\infty}$, we deduce that
$$| I^{(j)}_n|   \leq \|\nabla\varphi\|_{L^\infty} \Int_{0}^\infty \Big|{\psi^{(j)}}'(s)\Big|^2\,{\rm
e}^{- \alpha_n^{(j)} s }\,ds,$$
which ensures the result.
\end{rems}

\subsection{Main results}
The heart of this work is to  prove that the lack of compactness of the Sobolev embedding
$H^1(\R^2)\hookrightarrow{\mathcal L}(\R^2)$ can be reduced up to cores to the  example by Moser. Before coming to the statement of the main theorem, let us introduce some definitions as in \cite{BMM} and \cite{Ge2}.
\begin{defi}
\label{orthogen} A  scale is a sequence $\underline{\alpha}:=(\alpha_n)$
of positive real numbers going to infinity,   a core is a sequence $\underline{x}:=(x_n)$  of points in $\R^2$ and a profile is a  function $\psi$ belonging to the set
$$
{\cP}:=\Big\{\;\psi\in L^2(\R,{\rm e}^{-2s}ds);\;\;\; \psi'\in
L^2(\R),\;\psi_{|]-\infty,0]}=0\,\Big\}.
$$
Given two scales $\underline{\alpha}$, $\underline{\tilde \alpha}$,  two cores $\underline{x}$, $\underline{\tilde x}$ and two profiles $\psi$,  ${\tilde \psi}$, we  shall say that the triplets  $(\underline{\alpha},\underline{x}, \psi)$ and $(\underline{\tilde \alpha},\underline{\tilde x},{\tilde \psi})$ are orthogonal (in short $(\underline{\alpha},\underline{x}, \psi)\perp(\underline{\tilde \alpha},\underline{\tilde x},{\tilde \psi})$) if
\begin{equation}\label{caseI}  \mbox{either}\qquad \Big|\log\left({\tilde \alpha_n}/{\alpha_n}\right)\Big|\longrightarrow \infty \,,\end{equation}
or $  \tilde \alpha_n =  \alpha_n $ and
\begin{equation}\label{caseII} - \frac{ \log|x_n- \tilde x_n|}{\alpha_n} \longrightarrow a \geq 0  \,\, \mbox{with} \,\,\psi \,\, \mbox{or} \,\, {\tilde \psi}\,\, \mbox{null for} \,\, s < a\,.\end{equation}
\end{defi}

 Our  result  states as follows.
\begin{thm}
\label{noradmain} Let $(u_n)$ be a bounded sequence in
$H^1(\R^2)$ such that \bq \label{noradmain-assum1}
u_n\rightharpoonup 0, \eq \bq \label{noradmain-assum2}
\limsup_{n\to\infty}\|u_n\|_{\mathcal L}=A_0 >0 \quad \quad
\mbox{and}\eq
\bq \label{noradmain-assum4} \lim_{R\to\infty}\;
\limsup_{n\to\infty}\,\|u_n\|_{\mathcal L (|x|>R)}=0. \eq Then, there
exist a sequence of scales $(\underline{\alpha}^{(j)})$, a  sequence of cores $(\underline{x}^{(j)})$ and a sequence of profiles $(\psi^{(j)})$  such that the triplets $(\underline{\alpha}^{(j)}, \underline{x}^{(j)},\psi^{(j)})$ are pairwise
orthogonal and, up to a subsequence extraction, we have for all
$\ell\geq 1$, \bq \label{noraddecomp}
u_n(x)=\Sum_{j=1}^{\ell}\,\sqrt{\frac{\alpha_n^{(j)}}{2\pi}}\;\psi^{(j)}\left(\frac{-\log|x - x_n^{(j)}|}{\alpha_n^{(j)}}\right)+{\rm
r}_n^{(\ell)}(x),\quad\limsup_{n\to\infty}\;\|{\rm
r}_n^{(\ell)}\|_{\mathcal
L}\stackrel{\ell\to\infty}\longrightarrow 0. \eq
Moreover, we have
the following orthogonality equality \bq \label{ortogonal2} \|\nabla
u_n\|_{L^2}^2=\Sum_{j=1}^{\ell}\,\|{\psi^{(j)}}'\|_{L^2}^2+\|\nabla
{\rm r}_n^{(\ell)}\|_{L^2}^2+\circ(1),\quad n\to\infty. \eq
\end{thm}
\begin{rems}\quad\\ \label{mainrems}
\noindent$\bullet$ It may seem surprising that the elements involved in decomposition \eqref{noraddecomp}   are similar to those that characterize the lack of compactness in the radial case.  The idea behind it can be illustrated  by the following  anisotropic transform of  the
example  by Moser (More general examples will be treated in Appendix \ref{examples}).
\begin{lem}
\label{anisoexamp}
  Let us consider the sequence $(u_n)$ defined by
$$
u_n(x)= f_{\alpha_n}(\lambda_1\,x_1,\lambda_2\,x_2),
$$
where $f_{\alpha_n}$ is the  example by Moser and   $\lambda_i \in ]0, \infty[$, for $i \in \{1, 2\} $. Then
\bq
\label{exp1approx1}
u_n\,\asymp\,f_{\alpha_n}\quad\mbox{in}\quad\,\cL\,,
\eq
where the symbol
$\asymp$ means that the difference goes to zero in $\cL$  as $n$ tends to infinity.
\end{lem}
\begin{proof}[Proof of Lemma \ref{anisoexamp}]
To go to the proof of Lemma \ref{anisoexamp}, let us first consider the case where  $\lambda_1=\lambda_2 = \lam >0$, and set
$$
v_n(x)=\sqrt{\frac{\alpha_n}{2\pi}}\;{\mathbf L}\Big(\frac{-\log\lam|x|}{\alpha_n}\Big)-\sqrt{\frac{\alpha_n}{2\pi}}\;{\mathbf L}\Big(\frac{-\log|x|}{\alpha_n}\Big)\,,
$$
where ${\mathbf L}$ is the  profile by Moser introduced in Remarks \ref{remsradial}.
A straightforward computation leads to
\beqn
\|\nabla v_n\|_{L^2}^2&=&2\|{\mathbf L}'\|_{L^2}^2-\frac{2}{\alpha_n}\int_0^\infty\,{\mathbf L}'\left(-\frac{\log\lam r}{\alpha_n}\right){\mathbf L}'\left(-\frac{\log r}{\alpha_n}\right)\,\frac{dr}{r}\\
&:=&2\|{\mathbf L}'\|_{L^2}^2-2 K_n\,.
\eeqn
The change of variable $s=-\frac{\log r}{\alpha_n}$ yields
$$
K_n=\int_{\R}\,{\mathbf L}'(s) {\mathbf L}'(s-\frac{\log \lam}{\alpha_n})\,ds\,.
$$
But,
$$ \Big| K_n - \int_{\R} {\mathbf L}'(s)^2 \,ds \Big|  \leq \|{\mathbf L}'\|_{L^2} \| \tau_{\frac{\log \lam}{\alpha_n}} {\mathbf L}' - {\mathbf L}'\|_{L^2} \,. $$
Since ${\mathbf L}'\in L^2(\R)$,  we have for $\lam \neq 1$ (the case $\lam = 1$ being trivial)
$$\| \tau_{\frac{\log \lam}{\alpha_n}} {\mathbf L}' - {\mathbf L}'\|_{L^2} \to 0  \quad \mbox{as} \quad n  \to \infty,$$ which implies  that the sequence $(K_n)$ tends to $\|{\mathbf L}'\|_{L^2}^2$ as $n$ goes to infinity. It follows that $v_n\to 0$ strongly in $H^1(\R^2)$ which leads to the desired conclusion in that case. \\

Let us now consider the general case where $\lambda_2 < \lambda_1$ and set  $$
w_n(x)=u_n(x)-\sqrt{\frac{\alpha_n}{2\pi}}\;{\mathbf L}\Big(\frac{-\log|x|}{\alpha_n}\Big).
$$
Using the fact that the profile ${\mathbf L}$ is increasing, we get
$$
w_n^1(x)\leq w_n(x)\leq w_n^2(x),
$$
where
\beqn
w_n^1(x)&=&\sqrt{\frac{\alpha_n}{2\pi}}\;{\mathbf L}\Big(\frac{-\log \lam_1|x|}{\alpha_n}\Big)-\sqrt{\frac{\alpha_n}{2\pi}}\;{\mathbf L}\Big(\frac{-\log|x|}{\alpha_n}\Big),\\
w_n^2(x)&=&\sqrt{\frac{\alpha_n}{2\pi}}\; {\mathbf L} \Big(\frac{-\log \lam_2|x|}{\alpha_n}\Big)-\sqrt{\frac{\alpha_n}{2\pi}}\;{\mathbf L}\Big(\frac{-\log|x|}{\alpha_n}\Big)\,.
\eeqn
We deduce that  for all $x \in \R^2$
$$ \big| w_n(x) \big| \leq \big| w_n^1(x) \big| + \big| w_n^2(x) \big|
\leq \Big|\big| w_n^1(x) \big| + \big| w_n^2(x) \big| \Big|. $$
Taking advantage of invariance under modulus of the Orlicz norm \eqref{moeql} and the
 monotonicity  property  (see Lemma \ref{monoto}),  we conclude  the proof of the lemma since the sequences $(w_n^1)$ and $(w_n^2)$ converge strongly to $0$ in $H^1(\R^2)$. \\ \end{proof}
\noindent$\bullet$ Let us point out that the elementary concentrations $$g_n^{(j)}(x) := \sqrt{\frac{\alpha_n^{(j)}}{2\pi}}\;\psi^{(j)}\left(\frac{-\log|x - x_n^{(j)}|}{\alpha_n^{(j)}}\right)$$ are  completely different from the profiles involved in the characterization of the lack of compactness of the Sobolev embedding investigated in \cite{BCG}, \cite{Ge2} and \cite{jaffard}. In fact, we can prove that  for any $0< a < b$  and any sequence $(h_n)$ of nonnegative real numbers
\begin{equation}
\label{withscales}\int_{a < h_n |\xi|< b}  | \wh {\nabla g_n^{(j)}}(\xi)|^2 \,d\xi \to 0,\quad n\to\infty. \end{equation}
Actually, the scales $\alpha_n^{(j)}$ do not correspond to scales in the point of view of frequencies  as for the profiles  describing the lack of compactness of $\dot H^s(\R^d)\hookrightarrow L^p(\R^d)$ in \cite{Ge2} or  of $\dot H^1(\R^2)\hookrightarrow BMO(\R^2)$ in \cite{BCG}. It was emphasized in  \cite{Ge2} that the characteristic of being without  scales in the sense of \eqref{withscales} is measured using   the Besov norm $\dot B_{2,\infty}^0$ (a precise definition of  Besov spaces is available for instance in \cite{BCD}). We deduce that
$$ \|g_n^{(j)}\|_{\dot B_{2,\infty}^1}\to 0, \quad  n \to \infty.$$
\noindent$\bullet$ Now in our context, the scales correspond   to values taken by the functions $g_n^{(j)}$ in consistent sets  of  size. More precisely, saying that $\alpha_n$  is a scale for $u_n$ means that $u_n \geq c \sqrt{\alpha_n} $ on a set $E_n$ of Lebesgue measure greater than ${\rm e}^{-2\alpha_n}$.\\
\noindent$\bullet$ It will be useful later on to observe that if $ \|w_n\|_{\mathcal
L}\stackrel{n\to\infty}\longrightarrow 0 $, then for any scale $\alpha_n$, any positive constant $C$ and any ball $B(x_n,{\rm e}^{-\alpha_n})$ the sets
$$ F_n :=\Big\{x\in \R^2;\,|w_n(x)|\,\geq C \sqrt{\alpha_n} \Big\}$$ satisfy
\begin{equation}
\label{usfcond}\frac{|F_n \cap B (x_n, {\rm e}^{- \alpha_n })|}{|B (x_n, {\rm e}^{- \alpha_n })|} \longrightarrow 0,\quad \mbox{as}  \quad n\to\infty, \end{equation}
where ~$|\,\cdot \, |$   denotes the Lebesgue measure. Indeed  if $ \|w_n\|_{\mathcal
L}\stackrel{n\to\infty}\longrightarrow 0 $, then for any $\varepsilon>0$ there exists $N \in \N$ so that for $n \geq N$
$$ \Int_{\R^2}\,\left({\rm
e}^{\frac {|w_n (x)|^2}{\varepsilon^2}}-1\right)\,dx \leq \kappa,$$
which implies that
$$ |F_n \cap B (x_n, {\rm e}^{- \alpha_n })| \lesssim {\rm
e}^{- \frac {C^2}{\varepsilon^2} \alpha_n}$$ and ensures the result if $\varepsilon$ is chosen small enough.\\

More generally,  if for $n$ large $ \|w_n\|_{\mathcal
L}\leq \eta $, then for any scale $\alpha_n$, any positive constant $C > \sqrt{2}\eta$ and any ball $B(x_n,{\rm e}^{-\alpha_n})$
$$\frac{|F_n \cap B (x_n, {\rm e}^{- \alpha_n })|}{|B (x_n, {\rm e}^{- \alpha_n })|} \longrightarrow 0,\quad \mbox{as}  \quad n\to\infty.$$
\noindent$\bullet$ As well, if $ \|u_n\|_{\mathcal
L}\stackrel{n\to\infty}\longrightarrow  A_0 $, then for any scale $\alpha_n$ and any measurable set $S_n$ of Lebesgue measure $|S_n| \approx {\rm e}^{- c \,\alpha_n } $,   the sets
$$ G_n :=\Big\{x\in \R^2;\,|u_n(x)|\,\geq M \sqrt{\frac{\alpha_n}{2\pi}} \Big\},$$
with $M > \sqrt{4\pi c} A_0$
check
\begin{equation}\label{high} \frac{|G_n \cap S_n|}{|S_n|} \longrightarrow 0,\quad \mbox{as}  \quad n\to\infty.\end{equation}
\noindent$\bullet$ Orthogonality hypothesis means that the interaction between the elementary concentrations is negligible in the energy space. Roughly speaking,  condition \eqref{caseII} requires in the case where the cores are not  sufficiently distant, namely   in the case where $- \frac{ \log|x_n- \tilde x_n|}{\alpha_n} \stackrel{n\to\infty}\longrightarrow a > 0$,  the vanishing  of one of the profiles
 for  $s<a$. This condition is due to the fact that  the parts of the elementary concentrations respectively around the cores $x_n$ and $\tilde x_n$ resulting   from the profiles for the values $s<a$ interact.  Nevertheless,  the parts  coming from the profiles for the values $s>a$ do not see each other  and  that is why  the  vanishing requirement  in \eqref{caseII} applies only to  the zone $s<a$. \\
\noindent$\bullet$  The orthogonality assumption \eqref{caseII} only concerns the case where the limit  $a$ is finite. In effect, if the cores are close enough  in the sense that  $- \frac{ \log|x_n- \tilde x_n|}{\alpha_n} \stackrel{n\to\infty}\longrightarrow \infty$,  one   can add  them,
  up to a remainder term  tending to  zero in the energy space.  More precisely, if   we  consider the elementary concentration
$$ \sqrt{\frac{\alpha_n}{2\pi}}\;\varphi\left(\frac{-\log|x- x_n|}{\alpha_n}\right),$$
 under the assumption $- \frac{ \log|x_n|}{\alpha_n} \stackrel{n\to\infty}\longrightarrow \infty$, then by straightforward computations we get
 \begin{equation}
\label{ex1} \sqrt{\frac{\alpha_n}{2\pi}}\;\varphi\left(\frac{-\log|x- x_n|}{\alpha_n}\right)  \asymp \sqrt{\frac{\alpha_n}{2\pi}}\;\varphi\left(\frac{-\log|x|}{\alpha_n}\right)\quad \mbox{in} \quad  H^1. \end{equation}
\noindent$\bullet$ The hypothesis \eqref{noradmain-assum4} of compactness at infinity in the Orlicz space ensures the compactness at infinity in $L^2$ thanks to  \eqref{leborlicz}. It is necessary to eliminate the lack of compactness of $ H^1(\R^2)\hookrightarrow \cL(\R^2)$ due  to translations to infinity mentioned above and illustrated by sequences of type $u_n(x)=\varphi(x+x_n)$ where
$0\neq\varphi\in{\mathcal D}$ and $|x_n|\to\infty$. Such hypothesis allows us, up to a remainder small enough in $\cL $, to reduce to the case where the mass concentrations are situated in a fixed ball. In particular, for any fixed $j$, we can easily prove that the core $(\underline{x}^{(j)})$ is a bounded sequence.\\
\noindent$\bullet$ Concerning the behavior of  the Orlicz norm, we have the following result:
\begin{prop}
\label{sumOrlicz2} Let $(\underline{\alpha}^{(j)}, \underline{x}^{(j)},\psi^{(j)})_{1\leq
j\leq\ell}$  be a family of triplets of scales, cores  and profiles such that the scales are pairwise
orthogonal\footnote{As we will see later in Appendix \ref{proofprop{sumOrlicz2}}, property \eqref{OrliczMax1} fails for the same scale and the pairwise orthogonality of the couples $(\underline{x}^{(j)},\psi^{(j)})$. }, and set
$$
g_n(x)=\Sum_{j=1}^{\ell}\,\sqrt{\frac{\alpha_n^{(j)}}{2\pi}}\;\psi^{(j)}\left(\frac{-\log|x- x_n^{(j)}|}{\alpha_n^{(j)}}\right):=\Sum_{j=1}^{\ell}\,g_n^{(j)}(x)\;.
$$ Then
\begin{equation}
\label{OrliczMax1} \|g_n\|_{\cL}\to \sup_{1\leq
j\leq\ell}\,\left(\lim_{n\to\infty}\,\|g_n^{(j)}\|_{\cL}\right)\; = \sup_{1\leq
j\leq\ell}\,\left( \frac{1}{\sqrt{4\pi}}\,\max_{s>0}\;\frac{|\psi^{(j)}(s)|}{\sqrt{s}}\right).
\end{equation}
\end{prop}

\noindent$\bullet$ Let us mention that M. Struwe in \cite{Struwe88} studied the loss of compactness for the functional
$$
E(u)=\frac{1}{|\Omega|}\,\Int_{\Omega}\;{\rm e}^{4\pi |u|^2}\,dx,
$$
where $\Omega$ is a bounded domain in $\R^2$.  Also, weak continuity properties of this functional was recently investigated by Adimurthi and K. Tintarev  in \cite{AT2}.

\end{rems}


\subsection{Layout of the paper}

The  paper is organized as follows:  in Section \ref{mainnoradproofth}, we describe
 the algorithmic construction of the decomposition of a bounded sequence  $(u_n)_{n \in \N } $ in $ H^1(\R^2)$, up to a
subsequence extraction, in terms of asymptotically orthogonal
profiles concentrated around cores  of type~$\sqrt{\frac{\alpha_n}{2\pi}}\;\psi(\frac{-\log|x - x_n|}{\alpha_n})$ and prove Theorem \ref{noradmain}. In the appendix  \ref{appendix}, we deal  with
several complements for the sake of completeness: Appendix \ref{examples} is dedicated to the study of  significant examples, Appendix \ref{proofprop{sumOrlicz2}} to the proof of Proposition \ref{sumOrlicz2} specifying the behavior of the decomposition by means of profiles with respect to the Orlicz norm and  finally Appendix \ref{reacap} to the collection of all useful  known results  on rearrangement of functions and the notion of capacity  which are used in this text.\\

Finally, we mention that, $C$ will be used to denote a constant
which may vary from line to line. We also use $A\lesssim B$ (respectively $A\gtrsim B$) to
denote an estimate of the form $A\leq C B$ (respectively $A \geq C B$) for some absolute
constant $C$ and $A\approx B$ if $A\lesssim B$ and $A \gtrsim B$.
For simplicity, we shall also still denote by $(u_n)$ any
subsequence of $(u_n)$ and designate by $\circ(1)$ any sequence which tends to $ 0 $ as $n$ goes to infinity. For two scales  $(\alpha_n)$ and $(\beta_n)$, we shall notice  $(\alpha_n) \ll (\beta_n)$  if $\frac{\beta_n}{\alpha_n}\to\infty$.

\section{Proof of the main theorem}
\label{mainnoradproofth}


This section is devoted to the proof of Theorem \ref{noradmain}. As it is mentioned above, the analysis of the lack of compactness of Sobolev embedding of $H^1(\R^2)$ in the Orlicz space is different from the one conducted in the radial setting \cite{BMM}  where an $L^ \infty$  estimate far away from the origin is available. Such $L^ \infty$ estimate is obviously not valid in the general case of $H^1(\R^2)$ even far away  from a discrete  set. To be convenience, it suffices to consider the following bounded sequence in $H^1(\R^2)$
$$  u_n(x) = \sum_{k=1}^\infty\, \frac {1}{k^2} f_{\alpha_n} (x-x_k), $$
where $(\alpha_n)$ is a scale,  $f_{\alpha_n}$ is the  example by Moser and $(x_k)$ is a  sequence in $\R^2$ that has accumulation points.

\subsection{The radial case}\label{method}
Before going into the details of the proof of Theorem \ref{noradmain}, let us briefly recall the basic idea of the proof of Theorem  \ref{main} (for more details, one can consult \cite{BMM}): through  a diagonal subsequence extraction, the main step consists to extract a  scale  $(\alpha_n)$ and a profile $\psi$ such that
$$
 \|\psi'\|_{L^2}\geq C\,A_0,
$$
where $C$ is a universal constant. The extraction of the scale follows  from the fact that for any $ \varepsilon > 0 $
\begin{equation}
\label{heartest} \sup_{s\geq
0}\left(\Big|\frac{v_n(s)}{A_0-\varepsilon}\Big|^2-s\right)\to\infty,\quad
n\to\infty,
\end{equation}
with $v_n(s) =  u_n ({\rm e}^{- s})$.  Property \eqref{heartest} is proved by contradiction assuming that
$$ \sup_{s\geq 0,\,
n\in\N}\;\;\left(\Big|\frac{v_n(s)}{A_0-\varepsilon}\Big|^2-s\right)\leq
C<\infty,
$$
which ensures by virtue of Lebesgue theorem that
$$
\int_{|x|<1}\;\left({\rm
e}^{|\frac{u_n(x)}{A_0-\varepsilon}|^2}-1\right)\,dx=2\pi\,\int_{0}^\infty\;\left({\rm
e}^{|\frac{v_n(s)}{A_0-\varepsilon}|^2}-1\right)\,{\rm e}^{-2s}\,ds\to
0,\quad n\to\infty.
$$
Furthermore,  taking advantage of the radial estimate  \eqref{radest}, we deduce that  the sequence $(u_n)$ is   bounded on the set $\{|x|\geq 1\}$ which
implies that
$$
\int_{|x|\geq 1}\;\left({\rm
e}^{|\frac{u_n(x)}{A_0-\varepsilon}|^2}-1\right)\,dx\leq C\|u_n\|_{L^2}^2\to
0\,.
$$
In conclusion, this leads  to
$$
\Limsup_{n\to\infty}\,\|u_n\|_{\mathcal L}\leq A_0-\varepsilon,
$$
which is in contradiction with hypothesis \eqref{main-assum2}. Fixing $\varepsilon = A_0 / 2$,  a scale $(\alpha_n)$  can be extracted such that
$$
\frac{ A_0}{2}\sqrt{\alpha_n}\leq| v_n(\alpha_n)|\leq C
\sqrt{\alpha_n}+\circ(1).
$$
Finally, setting
$$
\psi_n(y)=\sqrt{\frac{2\pi}{\alpha_n}}\;v_n(\alpha_n y),
$$
one can prove that $\psi_n$ converges  simply to a profile $\psi$. Since $ |\psi_n(1)| \geq C  A_0$, we obtain
$$
C  A_0 \leq |\psi(1)|=\Big|\int_0^1\psi'(\tau)\,d\tau\Big|\leq
\|\psi'\|_{L^2(\R)},
$$
which ends the proof of the main point. \\

Our approach to characterize the lack of compactness in the general case of the Sobolev embedding
$$ H^1(\R^2)\hookrightarrow \cL(\R^2)$$
is
entirely different  from the one conducted in the radial case and uses in a crucial way capacity arguments. \\

As it was highlighted in Remarks \ref{mainrems}, the lack of compactness of the  Sobolev embedding in the Orlicz space  is due to large values of the sequences on sets of significant  Lebesgue measure. The main difficulty consists in  extracting  the cores, which are the points about which enough mass is concentrated.   To do so,  we will use some  capacity arguments and demonstrate by contradiction that
 the mass responsible for the lack of compactness can not be  dispersed.\\

The proof  of Theorem \ref{noradmain} is done in four  steps. In the first
step, using  Schawrz symmetrization, we first  study   $u_n^*$,
 the symmetric decreasing rearrangement of $u_n$
 (for an introduction to the subject see Appendix \ref{reacap}). Since $u_n^* \in H^1_{rad}(\R^2)$ and satisfies   assumptions of Theorem \ref{main},  it can be written as an orthogonal asymptotic decomposition by means  of elementary concentrations similar to examples of Moser, namely  of   type $ g_n^{(j)}(x)= \sqrt{\frac{\alpha_n^{(j)}}{2\pi}}\;\varphi^{(j)}\left(\frac{-\log|x|}{\alpha_n^{(1)}}\right)$. This actually yields the scales of
$u_n$ too.  Then,   in the second step, taking advantage of  \eqref{profile1}, we
reduce ourselves to one scale
and extract  the first core $(x_n^{(1)})$
and the first profile $\psi^{(1)}$ which leads to the extraction of the first  element $ \sqrt{\frac{\alpha_n^{(1)}}{2\pi}}\;\psi^{(1)}\left(\frac{-\log|x - x_n^{(1)}|}{\alpha_n^{(j)}}\right)$. This step constitutes  the heart of the proof  and relies on capacity arguments: the basic idea is to show that the mass does not scatter and is  mainly concentrated around some points  that will constitute the cores.
 In the third step, we study  the remainder term. If the limit of its Orlicz norm is null we stop the process. If not, we prove that this remainder term satisfies the same properties as the
sequence we  start with  which allows us to  extract a second  elementary concentration concentrated around a second core $(x_n^{(2)})$. Thereafter, we establish the property of orthogonality
between the  first  two elementary concentrations and finally we prove that this process converges.


\subsection{Extraction of the scales}
\label{Extractscales}

Using  Schawrz symmetrization, we first  study   $u_n^*$,  the symmetric decreasing rearrangement of $u_n$,  (see Appendix \ref{reacap} for all details). Since $u_n^* \in H^1_{rad}(\R^2)$ and satisfies by virtue of Proposition \ref{norminvar} the assumptions of Theorem \ref{main},  there
exists a sequence $(\underline{\alpha}^{(j)})$ of pairwise
orthogonal scales and a sequence of profiles $(\varphi^{(j)})$ in
${\cP}$ such that, up to a subsequence extraction, we have for all
$\ell\geq 1$, $$
u_n^*(x)=\Sum_{j=1}^{\ell}\,\sqrt{\frac{\alpha_n^{(j)}}{2\pi}}\;\varphi^{(j)}\left(\frac{-\log|x|}{\alpha_n^{(j)}}\right)+{\rm
r}_n^{(\ell)}(x),\quad\limsup_{n\to\infty}\;\|{\rm
r}_n^{(\ell)}\|_{\mathcal
L}\stackrel{\ell\to\infty}\longrightarrow 0.$$
Since $u_n^*$ is nonnegative, it is clear under the proof sketch presented in Section \ref{method} that the profiles $\varphi^{(j)}$ are nonnegative.
Moreover, thanks to \eqref{profile1} we can suppose that
$$ A_0 = \lim_{n\to\infty}\left\|  \sqrt{\frac{\alpha_n^{(1)}}{2\pi}}\;\varphi^{(1)}\left(\frac{-\log|x|}{\alpha_n^{(1)}}\right)  \right\|_{\mathcal L}$$
with $\displaystyle \max_{s>0}\;\frac{\varphi^{(1)}(s)}{\sqrt{s}} =  \varphi^{(1)}(1)$. It follows clearly that $ \varphi^{(1)}(1) = \sqrt{4\pi}\, A_0 $ and \begin{equation}
\label{est0} \varphi^{(1)}(s) \leq \sqrt{s} \sqrt{4\pi}\, A_0, \,\, \forall s \geq 0.\end{equation} Consequently, for any $0 <  \varepsilon <  1$ and $s\leq 1-\varepsilon$, we get that
\begin{equation}
\label{est1}\varphi^{(1)}(s) \leq (1- \frac \varepsilon 2) \sqrt{4\pi}\, A_0.
\end{equation}
Now, let us consider $0< \varepsilon_0< \frac 1 2 $  to be fixed later on and denote by $E^*_n$ the set
$$ E^*_n :=\Big\{x\in \R^2;\, \,g_n^{(1)}(x)\,\geq \sqrt{{2\alpha_n^{(1)}}}\;\Big(1- \frac {\varepsilon_0} {10}\Big) \, A_0 \Big\}, $$ where $g_n^{(1)}(x) = \sqrt{\frac{\alpha_n^{(1)}}{2\pi}}\;\varphi^{(1)}\left(\frac{-\log|x|}{\alpha_n^{(1)}}\right) $. Then, we have the following useful result: \begin{lem}
\label{estset1} There exists an integer $N_0$ such that for $n \geq N_0$
\begin{equation}
\label{noradhearteq0}
 |E^*_n | \geq {\rm e}^{-2\alpha_n^{(1)}},\end{equation} where ~$|\cdot |$ denotes the Lebesgue measure.
 \end{lem} \begin{proof} By definition, saying that $x$ belongs to $E^*_n$ is equivalent to saying that
 $$  \varphi^{(1)}\left(\frac{-\log|x|}{\alpha_n^{(1)}}\right) \geq  \Big(1- \frac {\varepsilon_0} {10}\Big) \sqrt{4\pi}\, A_0.  $$ Since $\varphi^{(1)}$  is continuous,   there exists $\eta > 0$ such that
 $$|s- 1| \leq \eta \Rightarrow \Big| \,\varphi^{(1)}(s)- \varphi^{(1)} (1) \Big| \leq \frac {\varepsilon_0} {10} \sqrt{4\pi}\, A_0.$$ Knowing that  $ \varphi^{(1)}(1) =  \sqrt{4\pi}\, A_0$, we deduce that  $ \varphi^{(1)}(s)  \geq \Big(1- \frac {\varepsilon_0} {10}\Big) \sqrt{4\pi}\, A_0$ provided that  $|s- 1| \leq \eta$. In other words, for  $ {\rm e}^{- \alpha_n^{(1)} (1 + \eta)} \leq |x| \leq {\rm e}^{- \alpha_n^{(1)} (1 - \eta)} $, we have $$ g_n^{(1)}(x)\geq \sqrt{\frac{\alpha_n^{(1)}}{2\pi}}\;\Big(1- \frac {\varepsilon_0} {10}\Big) \sqrt{4\pi}\, A_0.$$ This  achieves the proof of the lemma. \end{proof}
 \begin{rem}\label{estmeasure}
 It is clear that we have shown above  that there is $\eta > 0$ such that for $n$ big enough $|E^*_n | \gtrsim {\rm e}^{-2\alpha_n^{(1)}(1-\eta)}$. In fact in light of estimate \eqref{est1}, we have  $|E^*_n | \lesssim {\rm e}^{-2\alpha_n^{(1)}(1-\frac{\varepsilon_0}{5})}$.
 \end{rem}


\subsection{Reduction to one scale}
\label{redonescale}
Our aim now is  to  reduce to  one scale. For this purpose, we introduce, for any $0<a<M$, the odd cut-off function $\Theta_a^M$ as follows
\begin{eqnarray*}
\Theta_a^M(s)&=&\; \left\{
\begin{array}{cllll}0 \quad&\mbox{if}&\quad
0\leq\,s\leq a/2,\\\\ 2s-a \quad
&\mbox{if}&\quad a/2\leq s\leq a ,\\\\
s \quad&\mbox{if}&\quad a\leq s \leq M,\\\\
M\quad&\mbox{if}&\quad s\geq M.
\end{array}
\right.
\end{eqnarray*}
 The following result concerning  the truncation  of $u_n$ at the scale $\alpha_n^{(1)}$ will be crucial in the sequel:
\begin{prop}
\label{truncate} Let $(u_n)$ be a sequence in $H^1(\R^2)$ satisfying the assumptions of Theorem \ref{noradmain}, and define for fixed $0< a < M$,
$$ \wt{u}_{n,a}^M:= \sqrt{\frac{\alpha_n^{(1)}}{2\pi}}\,\Theta^M_a \left(\sqrt{\frac{2\pi}{\alpha_n^{(1)}}}\,u_n\right).$$
Then  for any $\varepsilon > 0$, there exist   $N \in \N$
and  $0< a < M$
such that
$$\left(\wt{u}_{n,a}^M\right)^*(x) = \sqrt{\frac{\alpha_n^{(1)}}{2\pi}}\; {\varphi}^{(1)}\left(\frac{-\log|x|}{\alpha_n^{(1)}}\right) +\wt{r}_n(x),$$
with $\|{\rm
\wt{r_n}}\|_{\mathcal
L} \leq \varepsilon$ for any $n \geq N$.
 \end{prop}
\begin{proof}
In light of Proposition \ref{norminvar}, $u_n^*$  satisfies hypothesis of Theorem \ref{main}  which allows us
  to  find  $\ell\geq 1$ such that
\beqn
u_n^*(x)&=&\Sum_{j=1}^{\ell}\,\sqrt{\frac{\alpha_n^{(j)}}{2\pi}}\;\varphi^{(j)}\left(\frac{-\log|x|}{\alpha_n^{(j)}}\right)+{\rm
r}_n^{(\ell)} (x),\quad\limsup_{n\to\infty}\;\|{\rm
r}_n^{(\ell)} \|_{\mathcal
L}
\leq \frac\eps2 \\
&=&\Sum_{j=1}^{\ell}\,g_n^{(j)}(x)+{\rm r}_n^{(\ell)}(x),
\eeqn
where $(\underline{\alpha}^{(j)})$ is a sequence of pairwise orthogonal scales and $(\varphi^{(j)})$  a sequence of profiles in
${\cP}$.\\

Since the cut-off function $\Theta_a^M$ is non-decreasing, it comes in view of \eqref{rearang}
\beqn
\left(\wt{u}_{n,a}^M\right)^*(x)&=& \sqrt{\frac{\alpha_n^{(1)}}{2\pi}}\,\Theta^M_a \left(\sqrt{\frac{2\pi}{\alpha_n^{(1)}}}\,u_n^*(x)\right)\\
&=& \sqrt{\frac{\alpha_n^{(1)}}{2\pi}}\,\Theta^M_a \left(\sqrt{\frac{2\pi}{\alpha_n^{(1)}}}\left(\Sum_{j=1}^{\ell}\,g_n^{(j)}(x)+{\rm r}_n^{(\ell)}(x)\right)\right).
\eeqn
We are  then reduced to the proof of  the following lemma.
\begin{lem}
\label{reduction}
Let $(\underline{\alpha}^{(j)})_{1\leq j\leq\ell}$  be a family of pairwise orthogonal scales, $(\varphi^{(j)})_{1\leq j\leq\ell}$ a family of profiles, and set
$$
v_n(x)=\sum_{j=1}^\ell\,\sqrt{\frac{\alpha_n^{(j)}}{2\pi}}\,\varphi^{(j)}\left(-\frac{\log|x|}{\alpha_n^{(j)}}\right)
 +{\rm r}_n^{(\ell)}(x),
$$
 where $\limsup_{n\to\infty}\; \|{\rm r}_n^{(\ell)}\|_{\cL} \leq \frac\eps2$.
Then, for any $1\leq k\leq\ell$, we have (as $n\to\infty$),
\begin{equation}
\label{Orl0}
\limsup_{n\to\infty}\; \left\|\sqrt{\frac{\alpha_n^{(k)}}{2\pi}}\,\Theta_a^M\left(\sqrt{\frac{2\pi}{\alpha_n^{(k)}}}\,v_n(x)\right)-
\sqrt{\frac{\alpha_n^{(k)}}{2\pi}}\,\Theta_a^M\left(\varphi^{(k)}\left(-\frac{\log|x|}{\alpha_n^{(k)}}\right)\right)\right\|_{\cL}
 \leq \frac\eps2 .
\end{equation}
In particular, the profile associated to $\sqrt{\frac{\alpha_n^{(k)}}{2\pi}}\,\Theta_a^M\left(\sqrt{\frac{2\pi}{\alpha_n^{(k)}}}\,
 u_n^*\right)$ is $\Theta_a^M \circ \varphi^{(k)}$.
\end{lem}
\begin{proof} For simplicity, we assume that $\ell=2$ and write
$$
v_n(x)=\sqrt{\frac{\alpha_n}{2\pi}}\,\varphi\left(-\frac{\log|x|}{\alpha_n}\right)+\sqrt{\frac{\beta_n}{2\pi}}\,\psi\left(-\frac{\log|x|}{\beta_n}\right)+{\rm r}_n (x)\,,
$$
where $(\alpha_n)\perp (\beta_n)$ are two scales, $\varphi, \psi$ are two profiles, and $\|{\rm r}_n\|_{\cL}\to 0$ as $n\to\infty$.\\

Let $\varepsilon>0$ to be chosen later. It suffices to prove that,
for some $\lambda=\lambda(\varepsilon)\to 0$ as $\eps$ goes to zero, we have (for $n$ big enough)
\bq
\label{orlicz0}
\int_{\R^2}\,\left({\rm e}^{|\frac{g_n(x)}{\lambda}|^2}-1\right)\,dx\leq \kappa,
\eq
where\footnote{For simplicity, we write $\Theta$ instead of $\Theta^M_a$.}
$$
g_n(x)=\sqrt{\frac{\alpha_n}{2\pi}}\left[\Theta\left(\varphi\left(-\frac{\log|x|}{\alpha_n}\right)+
\sqrt{\frac{\beta_n}{\alpha_n}}\,\psi\left(-\frac{\log|x|}{\beta_n}\right)+\sqrt{\frac{2\pi}{\alpha_n}}\,{\rm r}_n(x)\right)-\Theta\left(\varphi\left(-\frac{\log|x|}{\alpha_n}\right)\right)\right]\,.
$$

Since $\frac{|\psi(s) |  + |\varphi(s)|  }{\sqrt{s}}$ goes to zero as $s$ tends to
 either  $0$ or  $\infty$, there exists $0<s_0<S_0$ such that
\bq
\label{phipsi1}
     |\psi(s)| + |\varphi(s)|  \leq \varepsilon\sqrt{s},\quad s\in [0,s_0]\cup [S_0,\infty[.
\eq

$ $ 
 We first take the case
\bq
\label{cas1}
\frac{\beta_n}{\alpha_n}\longrightarrow \infty\,.
\eq
For $\lambda>0$, we have
\beqn
\int_{\R^2}\,\left({\rm e}^{|\frac{g_n(x)}{\lambda}|^2}-1\right)\,dx&=&\int_{\cC_n}\,\left({\rm e}^{|\frac{g_n(x)}{\lambda}|^2}-1\right)\,dx+\int_{\R^2\backslash\cC_n}\,\left({\rm e}^{|\frac{g_n(x)}{\lambda}|^2}-1\right)\,dx\\&=& {\mathbf I}_n+{\mathbf J}_n,
\eeqn
where $\cC_n=\Big\{ |x|\leq {\rm e}^{-s_0\beta_n}\,\Big\}$.

Notice, that for $x \in \R^2\backslash \cC_n$, we have
$ \Big|  \sqrt{\frac{\beta_n}{2\pi}}\,\psi\left(-\frac{\log|x|}{\beta_n}\right)   \Big| \leq \frac{\eps}{\sqrt{2\pi}}
 \sqrt{- \log|x| } $ and hence using that $|\Theta'| \leq 2$, we get
for  $x \in\R^2\backslash  \cC_n$:
$$
|g_n(x)| \leq 2 \frac{\eps}{\sqrt{2\pi}}\sqrt{- \log|x| }  + 2 |r_n (x)|.
$$
Taking advantage of the fact that the elementary concentration $\sqrt{\frac{\beta_n}{2\pi}}\,\psi\left(-\frac{\log|x|}{\beta_n}\right)$ is supported in the unit ball $B_1$, we deduce that
\beqn
{\mathbf J}_n&\leq &\int_{B_1\backslash  \cC_n}\,\left({\rm e}^{ 8 \frac{\eps^2}{{2\pi} \lambda^2}
 | \log|x| |  } -1 \right) \,dx +
\int_{  \R^2  \backslash  \cC_n}  \left( {\rm e}^{  8 \frac{|r_n (x)|^2}{\lambda^2}  }  -1\right)\,dx\\
&\leq& \kappa,
\eeqn
for $n$ sufficiently large  and $\eps$ small enough satisfying $ 8\eps^2 < 2\pi \lambda^2 $.\\

To control ${\mathbf I}_n $, we use that $\Theta $ is bounded:
\beqn
{\mathbf I}_n&\leq &\int_{\cC_n}\,\left({\rm e}^{\frac{2M^2\alpha_n}{\pi\lambda^2}}-1\right)\,dx\\
&\leq&\pi\left({\rm e}^{\frac{2M^2\alpha_n}{\pi\lambda^2}}-1\right)
{\rm e}^{-2s_0\beta_n}\to 0.
\eeqn

 Now, we  assume that
\bq
\label{cas2}
\frac{\beta_n}{\alpha_n}\longrightarrow 0\,.
\eq

For $\lambda>0$, we have
\beqn
\int_{\R^2}\,\left({\rm e}^{|\frac{g_n(x)}{\lambda}|^2}-1\right)\,dx&=&\int_{\cD_n}\,\left({\rm e}^{|\frac{g_n(x)}{\lambda}|^2}-1\right)\,dx+\int_{\R^2\backslash\cD_n}\,\left({\rm e}^{|\frac{g_n(x)}{\lambda}|^2}-1\right)\,dx\\&=& {\mathbf I}_n+{\mathbf J}_n,
\eeqn
where $\cD_n=\Big\{ |x|\leq {\rm e}^{-S_0\beta_n}\,\Big\}$. \\

Notice, that for $x \in  \cD_n$, we have
$ \Big|  \sqrt{\frac{\beta_n}{2\pi}}\,\psi\left(-\frac{\log|x|}{\beta_n}\right)   \Big| \leq \frac{\eps}{\sqrt{2\pi}}
 \sqrt{- \log|x| } $ and hence using that $|\Theta'| \leq 2$, we get
for  $x \in   \cD_n$:
$$
|g_n(x)| \leq   2  \frac{\eps}{\sqrt{2\pi}}\sqrt{- \log|x| }  + 2 |r_n (x)|.
$$
It follows that
\beqn
{\mathbf I}_n&\leq &\int_{  \cD_n}\,\left({\rm e}^{ 8 \frac{\eps^2}{{2\pi} \lambda^2}
 | \log|x| |  } -1 \right) \,dx +
\int_{    \cD_n}  \left( {\rm e}^{  8 \frac{|r_n (x)|^2}{\lambda^2}  }  -1\right)\,dx\\
&\leq& \kappa,
\eeqn
for $n$ sufficiently big and $ \eps$ small enough so that $ 8\eps^2 < 2\pi \lambda^2 $.\\

To control the integral ${\mathbf J}_n $, we use that for  $x \in\R^2\backslash  \cD_n$ and $n$ large enough,
 we have   $\frac{-\log|x|}{ \alpha_n}  < S_0 \frac{\beta_n}{\alpha_n}  \to 0 $. Thus
$$\Big|\varphi\left(\frac{-\log|x|}{ \alpha_n}\right)\Big| \leq \eps \sqrt{ \frac{-\log|x|}{ \alpha_n}   }
 \leq \eps   \sqrt{ \frac{ S_0 \beta_n }{ \alpha_n}   }  .$$
Let $M_1 = \displaystyle\sup_{s< S_0} |\psi(s)| $. Using the fact that $\Theta$ is non-decreasing, $|\Theta(s)| \leq |s|$ and
$$
|\Theta(s_1+s_2+s_3)| \leq |\Theta(3s_1)| + |\Theta(3s_2)| + |\Theta(3s_3)|,
$$
we infer
$$
|g_n(x)| \leq \sqrt{\frac{\alpha_n}{2\pi}}\left[ 2  \Theta\Big( 3  \sqrt{\eps}   \sqrt{ \frac{ S_0 \beta_n }{ \alpha_n}   }  \Big)   +  \Theta\Big( 3 M_1  \sqrt{ \frac{   \beta_n }{ \alpha_n} }  \Big)  \right]   +   6|r_n (x)| = 6 |r_n (x)|,
$$
for $n$ large enough.  The conclusion follows at least for $l=2$.\\


To extend this proof to the case $l\geq 3$, we can without loss of generality assume that
$ \frac{\alpha_n^{(j)} }{ \alpha_n^{(j+1)}  }   \to  0  $ when $n$ goes to infinity, $1\leq j \leq l-1$.
 We also assume
that $1<k<l$ otherwise we need only one splitting as in the case $l=2$. Now, we have just to combine the two previous
cases by introducing
$$
\cC_n=\Big\{ |x|\leq {\rm e}^{-s_0  \alpha_n^{(k+1)}   }\,\Big\},\quad\mbox{and}\quad \cD_n=\Big\{ |x|\leq {\rm e}^{-S_0  \alpha_n^{(k-1)}   }\,\Big\},
$$
where
$0<s_0<S_0$ are  such that
\bq
\label{phipsij}
    |\varphi^{(j)}(s)|  \leq \varepsilon\sqrt{s},\quad s\in [0,s_0]\cup [S_0,\infty[,
\eq
 and split the integral into three parts: For $\lambda>0$, we denote
\beqn
\int_{\R^2}\,\left({\rm e}^{|\frac{g_n(x)}{\lambda}|^2}-1\right)\,dx&=&
\int_{\cC_n}\,\left({\rm e}^{|\frac{g_n(x)}{\lambda}|^2}-1\right)\,dx+
\int_{\cD_n\backslash\cC_n}\,\left({\rm e}^{|\frac{g_n(x)}{\lambda}|^2}-1\right)\,dx \\&+&
\int_{\R^2\backslash\cD_n}\,\left({\rm e}^{|\frac{g_n(x)}{\lambda}|^2}-1\right)\,dx
\\
&=& {\mathbf I}_n+{\mathbf J}_n +{\mathbf K}_n .
\eeqn
 The rest of the proof combines the previous two cases.
We have, therefore, completed the proof of Lemma \ref{reduction}.
\end{proof}

 $ $ 
 
Invoking Lebesgue theorem and Theorem 6.19 page 154 in \cite{Analysis}, we find that $\Theta^M_a ( {\varphi}^{(1)} ) $ goes to $ {\varphi}^{(1)} $ in $\mathcal P$ when $a$ goes to zero and $M$ goes to infinity. Indeed
\beqn
 & & \int_0^\infty\; \left|(\Theta^M_a)'\left({\varphi}^{(1)}(s)\right)-1\right|^2\,\left|{{\varphi}^{(1)}}'(s)\right|^2\,ds
  =\int_{\{{\varphi}^{(1)}(s)=0\}}\;\left|{{\varphi}^{(1)}}'(s)\right|^2\,ds \\&+&
  \int_{\{|{\varphi}^{(1)}(s)|>0\}}\;\left|(\Theta^M_a)'\left({\varphi}^{(1)}(s)\right)-1\right|^2
  \,\left|{{\varphi}^{(1)}}'(s)\right|^2\,ds\,.
  \eeqn 
 The first integral vanishes since ${{\varphi}^{(1)}}'(s)=0$ for almost $s$ in the set $\{{\varphi}^{(1)}(s)=0\}$, while the second integral can be dealt using the simple convergence of the sequence $\big((\Theta^M_a)'\left({\varphi}^{(1)}(s)\right)-1\big)$ to zero when $(a,M)$ goes to $(0,\infty)$ and the fact that $$\left|(\Theta^M_a)'\left({\varphi}^{(1)}(s)\right)-1\right|^2\,\left|{{\varphi}^{(1)}}'(s)\right|^2 \lesssim \left|{{\varphi}^{(1)}}'(s)\right|^2.$$
 Finally, as the function $\Theta^M_a$  vanishes at $0$, we easily get
 $$ \int_0^\infty\; \left|\Theta^M_a\left({\varphi}^{(1)}(s)\right)- {\varphi}^{(1)}(s)\right|^2\,{\rm e}^{-2s}\,ds \to 0\,. $$  \\ 
The proof of Proposition \ref{truncate} is then achieved.
\end{proof}

$ $

$ $ 

\subsection{Extraction of the cores and profiles}
\label{Extractcopro}
\subsubsection{Extraction of the first core}
\label{truncate-sub}
Due to the previous subsection and estimates \eqref{usfcond}-\eqref{high}, it is enough to make the
extraction from $ \wt{u}_{n,a}^M $. Ignoring the rest term,
 we are reduced to studying the case of a sequence   $u_n$  satisfying  \begin{equation} \label{case1}u_n^*(x) = \sqrt{\frac{\alpha_n^{(1)}}{2\pi}}\;\varphi^{(1)}\left(\frac{-\log|x|}{\alpha_n^{(1)}}\right)\end{equation}
and such that $ |\varphi^{(1)} |  $ is bounded by M. This assumption will only be made in
this subsection.
Our approach to extract cores and
profiles  relies on a diagonal subsequence extraction and the heart of
the matter is reduced to the proof of the following lemma:

\begin{lem}
\label{noradheart} Under the above notation, there exist  $\delta_0 > 0 $ and ~$N_1 \in \N $ such that for any $n \geq N_1$ there exists $ x_n $ such that
\begin{equation}
\label{noradhearteq}
 \frac {|E_n \cap B (x_n, {\rm e}^{- b \,\alpha_n^{(1)}} )|}{|E_n|} \geq \delta_0 A^2_0,\end{equation}
 where $ E_n :=\Big\{x\in \R^2;\,|u_n(x)|\,\geq \sqrt{ 2 \alpha_n^{(1)}}\;\Big(1- \frac {\varepsilon_0} {10}\Big) \, A_0 \Big\} $, $B (x_n, {\rm e}^{- b \,\alpha_n^{(1)}} )$ designates the ball of center $x_n$ and radius ${\rm e}^{- b \,\alpha_n^{(1)}} $ with $b= 1 - 2 \varepsilon_0$ and~$|\,\cdot \, |$  still denotes the Lebesgue measure.
 \end{lem}

\begin{proof}
 Let us  assume by contradiction that \eqref{noradhearteq} does not hold. Consequently, up to a subsequence  extraction, we have for any $\delta > 0$, ~$n \in \N $ and  $x \in \R^2$
\begin{equation}
\label{noradhearteqnon} \frac {|E_n \cap B (x, {\rm e}^{- b \,\alpha_n^{(1)}} ) |}{|E_n|} \leq \delta A^2_0\,.\end{equation} In particular, inequality  \eqref{noradhearteqnon} occurs for any ball centered at a point belonging to ${\mathbf T}_n := ({\rm e}^{- b \, \alpha_n^{(1)}} \Z)\times ({\rm e}^{- b \, \alpha_n^{(1)}} \Z)$. It will be useful later on to notice that the balls $B (x, {\rm e}^{- b\, \alpha_n^{(1)}} )$ constitute a covering of $\R^2$ when the point  $x$ varies in
 ${\mathbf T}_n$  and that each point of $\R^2$ belongs at most  to four balls among $$ {\cB}_n := \Big\{\,B (x, {\rm e}^{- b\, \alpha_n^{(1)}}),\; x \in {\mathbf T}_n\Big\}.$$ This implies in particular that
\begin{equation}
\label{estcap1} \|\nabla u_n\|_{L^2(\R^2)}^2 \geq \frac 1 4 \sum_{{\mathbf B} \in {\cB}_n} \|\nabla u_n\|_{L^2({\mathbf B})}^2\,.\end{equation}

Now, our goal is to get a contradiction by proving that for $\delta $ small enough the sum $ \frac 1 4 \sum_{{\mathbf B} \in {\cB}_n} \|\nabla u_n\|_{L^2({\mathbf B})}^2$ exceeds the energy of~$u_n$.\\

For this purpose, let us  first estimate the energy of $u_n$ on each ball  ${\mathbf B} \in {\cB}_n$, making  use of capacity arguments. To do so, we shall take advantage of the fact that the values of $|u_n|$ on ${\mathbf B}$ varies at least from $\sqrt{ 2 \alpha_n^{(1)}}\;\Big(1- \frac {\varepsilon_0} {10}\Big)\, A_0$ on $E_n \cap {\mathbf B}$ to $\sqrt{ 2 \alpha_n^{(1)}}\,\Big(1- \frac {\varepsilon_0} {2}\Big)\, A_0$ on a set of Lebesgue measure greater than~$\frac { |{\mathbf B} |}{2}$, for all ~$n\geq N_{\varepsilon_0}$, where $N_{\varepsilon_0}$ is an integer big enough which only depends on $\varepsilon_0$.

Indeed, by definition of ~$E_n$, we have
$$|u_{n_{|E_n \cap {\mathbf B}}}| \geq \sqrt{ 2 \alpha_n^{(1)}}\;\Big(1- \frac {\varepsilon_0} {10}\Big) \, A_0. $$
Besides, thanks to  \eqref{est0} and \eqref{est1},   we get that for any $s\leq 1-\varepsilon_0$
\begin{equation}\label{estim} \varphi^{(1)}(s) \leq (1- \frac {\varepsilon_0} 2) \sqrt{4\pi}\, A_0.\end{equation}
Thus, if we designate by $ H^*_n$ the set
$$ H^*_n :=\Big\{\;x\in \R^2;\;\;\; u^*_n(x) \geq \sqrt{ 2 \alpha_n^{(1)}}\;\Big(1- \frac {\varepsilon_0} {2}\Big) \, A_0 \Big\}, $$ we obtain in view of \eqref{case1} and \eqref{estim}  that $H^*_n \subset B (0, {\rm e}^{- (1-\varepsilon_0 ) \alpha_n^{(1)}}) $ which implies that  $|H^*_n | \leq  \pi {\rm e}^{- 2(1-\varepsilon_0 ) \alpha_n^{(1)}} $. We deduce, by virtue of Proposition \ref{rearr} that  the set
$$ H_n :=\Big\{\;x\in \R^2;\;\;\; \,|u_n(x)|\,\geq \sqrt{ 2 \alpha_n^{(1)}}\;\Big(1- \frac {\varepsilon_0} {2}\Big) \, A_0 \Big\} $$
is of Lebesgue measure $|H_n | = |H^*_n | \leq  \pi {\rm e}^{- 2(1-\varepsilon_0 ) \alpha_n^{(1)}}$.

Since $ |{\mathbf B}| =  \pi {\rm e}^{- 2(1-2\varepsilon_0 ) \alpha_n^{(1)}}$ and ~$\alpha_n^{(1)} \to \infty$ as $n$ goes to infinity,
there exists $N_{\varepsilon_0}$ (which only depends on $\varepsilon_0$) such that   the ball ${\mathbf B}$ contains a set $\wt {\mathbf B}_n$ on which we have
$ |u_n |\leq \sqrt{ 2 \alpha_n^{(1)}}\;\Big(1- \frac {\varepsilon_0} {2}\Big) \, A_0\,$ and so that ~$ |\wt {\mathbf B}_n | \geq \frac { |{\mathbf B} |}{2}$  for all ~$n\geq N_{\varepsilon_0}$.\\

To achieve the proof of Lemma \ref{noradheart} and get a contradiction, we will estimate the energy of $u_n$ on the set $\wt  {\cB}_n$  of balls  ${\mathbf B} \in {\cB}_n $ satisfying $| E_n \cap {\mathbf B} | \geq {\rm e}^{-10 \, \alpha_n^{(1)}}$. Indeed by virtue of \eqref{noradmain-assum4} and as it was point out in Remarks \ref{mainrems}, we can reduce to the case where concentrations occur only in a fixed ball $B(0, R_0)$. But, since we cover the ball $B(0, R_0)$ by at most a number of order ${\rm e}^{ 2\, b \,\alpha_n^{(1)}}$ of balls that are part of the set ${\cB}_n$, necessary mass concentrated in $E_n$ is mainly due to balls of $\wt  {\cB}_n$. In effect, the contribution of balls ${\mathbf B} \in {\cB}_n $ satisfying $| E_n \cap {\mathbf B} | \leq {\rm e}^{-10 \, \alpha_n^{(1)}}$  is at most equal to ${\rm e}^{-8 \, \alpha_n^{(1)}} $ which is a  negligible part of $| E_n|$  in view of Remark \ref{estmeasure} and Proposition \ref{rearr}.\\

Taking advantage of the fact that the values of $|u_n|$ on ${\mathbf B} \in \wt  {\cB}_n$ varies at least from the value $\sqrt{ 2 \alpha_n^{(1)}}\;\Big(1- \frac {\varepsilon_0} {10}\Big) \, A_0$ on  the set $E_n \cap {\mathbf B}$  of Lebesgue measure greater than ~${\rm e}^{-10 \, \alpha_n^{(1)}}$  to the value  $\sqrt{ 2 \alpha_n^{(1)}}\;\Big(1- \frac {\varepsilon_0} {2}\Big) \, A_0$ on  $\wt {\mathbf B}_n$ which is  of Lebesgue measure greater than ~$\frac { |{\mathbf B} |}{2}$  for ~$n\geq N_{\varepsilon_0}$,  it follows from Proposition \ref{mainextraccore} that
\begin{eqnarray}
\|\nabla u_n\|_{L^2({\mathbf B})}^2 &\geq& C \,\Big(\big(\frac {\varepsilon_0} {2} - \frac {\varepsilon_0} {10}\big)\sqrt{ 2 \alpha_n^{(1)}}\;\, A_0\Big)^2 \frac 1 {\log\left( \frac {{\rm e}^{-(1-2 \varepsilon_0)\alpha_n^{(1)}}}{\sqrt{|E_n \cap {\mathbf B}|}}\right)} \nonumber  \\ & \geq & C  \varepsilon^2_0\, A^2_0 \,,\label{energy1}
\end{eqnarray}
where $C$ is an absolute constant.\\

Hence, thanks to  \eqref{estcap1},  we obtain
 \begin{equation}\label{estgen}
 4\,\|\nabla u_n\|_{L^2(\R^2)}^2 \geq    \# (\wt  {\cB}_n) \,C \varepsilon^2_0\, A^2_0, \end{equation}
where $\# (\wt  {\cB}_n)$  denotes the cardinal of $\wt  {\cB}_n$. But  by  \eqref{noradhearteqnon},   the covering  of  $\R^2$ by $ {\cB}_n$ and the fact that mass concentrated in $E_n$ is mainly due to balls of $\wt  {\cB}_n$, we have necessary $$\# (\wt  {\cB}_n) \geq \frac 1 { 2 \delta \, A^2_0}$$ which  yields a contradiction for $\delta$ small enough in view of \eqref{estgen}.

\end{proof}
\subsubsection{Extraction of the first profile}

\label{first profile}

Let us set
\begin{equation}\label{R} \psi_n (y,\theta)= \sqrt{\frac{2\pi}{\alpha_n^{(1)}}}\;v_n(\alpha_n^{(1)} y,\theta ),\end{equation}
where~$v_n(s,\theta )= \Big(\tau_{- x_n^{(1)}}u_n\Big) ({\rm e}^{-s} \cos \theta , {\rm e}^{-s} \sin \theta)$ and $x_n^{(1)}$ satisfies
\begin{equation}\label{R0}
 \frac {|E_n \cap B (x_n^{(1)}, {\rm e}^{- (1 - 2 \varepsilon_0) \,\alpha_n^{(1)}} )|}{|E_n|} \geq \delta_0 A^2_0, \end{equation} with  $E_n$  defined by Lemma \ref{noradheart}.  Using  the invariance of  Lebesgue measure  under translations, we get
\begin{equation}\label{R1}
\|\nabla u_n\|_{L^2}^2= \frac {1}{2\pi} \int_{\R}\int_0^{2\pi}\,|\partial_y\psi_n (y,\theta)|^2dy d\theta+ \frac {(\alpha_n^{(1)})^2}{2\pi}\int_{\R}\int_0^{2\pi}\,|\partial_\theta\psi_n (y,\theta)|^2dy d\theta.
\end{equation}
Since ~$\alpha_n^{(1)}$ tends to infinity and ~$(u_n)$ is bounded in ~$H^1(\R^2)$,  \eqref{R1} implies that
\begin{eqnarray}\partial_\theta\psi_n & \to & 0  \label{R4}\quad \mbox{and}
\\
\partial_y \psi_n & \rightharpoonup & g  \label{R5},
\end{eqnarray}
up to a subsequence extraction, in~$L^2 (y,\theta)$ as ~$n$ tends to infinity. Moreover
\begin{equation} \label{negcomp}\psi_n  \to 0  \, \, \mbox{in} \,  \, L^2 (] - \infty, 0 ] \times [ 0, 2\pi])  \,  \,  \mbox{as}  \,  \,  n \to \infty.\end{equation}
Indeed, we have
\begin{eqnarray*}
\| u_n\|_{L^2}^2 & = & \int_{\R} \int^{2\pi}_{0} |v_n (s,\theta)|^2 {\rm
e}^{-2s}\,ds\,d\theta \\ & = & \frac{(\alpha_n^{(1)})^2}{2\pi} \int_{\R} \int^{2\pi}_{0} |\psi_n (y,\theta)|^2 {\rm
e}^{-2\alpha_n^{(1)} y}\,dy\,d\theta\\ &\geq& \frac{(\alpha_n^{(1)})^2}{2\pi} \int^0_{-\infty} \int^{2\pi}_{0} |\psi_n (y,\theta)|^2\,dy\,d\theta,
\end{eqnarray*}
which ends the proof of \eqref{negcomp}. \\

The following lemma summarizes the principle properties of the function $g$ given above by  \eqref{R5}.
\begin{lem}
\label{summerize}
The function~$g$ given by \eqref{R5} only depends on the variable~$y$ and is null on $] - \infty, 0 ] $.  Besides, we have
\begin{equation} \label{imprel} \frac  1 {2\pi} \int^{2\pi}_{0}  \Big(\psi_n (y_2,\theta)-\psi_n (y_1,\theta)\Big)\,d\theta \to  \int^{y_2}_{y_1} g(y)\,dy,\end{equation}
for any $y_1, y_2  \in \R$,  as ~$n$ tends to infinity.
\end{lem}
\begin{proof}[Proof of Lemma \ref{summerize}]
Let us go to the proof of the fact that  the function~$g$ only depends on the variable~$y$. First,
 by \eqref{R4},  $\partial_\theta\psi_n  \to  0$  in ~$L^2 (y,\theta)$ and hence   we have~$\partial_y (\partial_\theta\psi_n)  \to  0 \, \, \mbox{in} \, \, {\mathcal D}'$.  Second,  under  \eqref{R5}, $\partial_y \psi_n  \to  g  \, \, \mbox{in}\, \, {\mathcal D}'$ which implies that ~$\partial_\theta (\partial_y\psi_n) \to  \partial_\theta g  \, \, \mbox{in} \, \, {\mathcal D}'$. Therefore, we deduce that $\partial_\theta g=0$.  The fact that $g \equiv 0$ on $] - \infty, 0 ] $ derives from  \eqref{negcomp}.\\

Now, taking advantage of the fact that ~$\partial_y \psi_n  \rightharpoonup  g $ in ~$L^2$ as ~$n$ tends to infinity, we get for any ~$y_1 \leq y_2$,
$$ \langle \partial_y\psi_n, {\mathbf 1}_{[ y_1, y_2]}\rangle \to 2 \pi \int^{y_2}_{y_1} g(y)\,dy.$$
But,
\begin{eqnarray*}
\langle \partial_y\psi_n, {\mathbf 1}_{[ y_1, y_2]}\rangle & = & \int^{2\pi}_{0}  \int^{y_2}_{y_1}  \partial_y \psi_n (y,\theta)\,dy\,d\theta \\ & = &  \int^{2\pi}_{0}  (\psi_n (y_2,\theta)-\psi_n (y_1,\theta))\,d\theta.
\end{eqnarray*}
This leads to \eqref{imprel}.
\end{proof}
Before extracting the first profile,  let us begin by establish
 the following lemma:
\begin{lem}
\label{nonradcont}
The function $F_n(y)= \frac{1}{2\pi} \int^{2\pi}_{0} \psi_n (y,\theta)\, d\theta $ is continuous on ~$\R$.
\end{lem}
\begin{proof}[Proof of Lemma \ref{nonradcont}]
By Cauchy-Schwarz's inequality, we get for $y_1$ and $y_2$ in $\R$:
$$ \big|F_n (y_1)- F_n (y_2)\big|^2 \leq \frac{1}{2\pi} \int^{2\pi}_{0} \big|\psi_n (y_1,\theta)- \psi_n (y_2,\theta)\big|^2\, d \theta. $$
But,
\begin{eqnarray*} \Big|\psi_n (y_1,\theta)- \psi_n (y_2,\theta)\Big| &=& \Big| \int^{y_2}_{y_1} \partial_y \psi_n (\tau,\theta)\, d \tau\Big|\\ &\leq &  \sqrt{|y_1-y_2|} \Big(\int^{y_2}_{y_1} \big| \partial_y \psi_n (\tau,\theta)\big|^2\, d \tau\Big)^{\frac 1 2} \end{eqnarray*}
Therefore
$$\big|F_n (y_1)- F_n (y_2)\big|^2 \leq \frac{1}{2\pi} |y_1-y_2| \int^{2\pi}_{0} \int^{y_2}_{y_1} \big| \partial_y \psi_n (\tau,\theta)\big|^2\, d \tau \,d \theta \leq C |y_1-y_2|,$$
which implies that the sequence ~$(F_n)$ is uniformly in ~$n$ in the H\"older space ~$C^\frac{1}{2}(\R)$.
\end{proof}

Let us  now introduce the  function
\begin{equation}
\label{nonradprofile}
\psi^{(1)} (y)= \int^{y}_{0} g(\tau) \, d \tau.
\end{equation}
Our goal  in what follows is to prove the following  proposition:
\begin{prop}\label{propconvprofile}
The function $\psi^{(1)} $ defined by  \eqref{nonradprofile} belongs to the set of profiles
${\cP}$. Besides   for any $y \in \R$, we have
\begin{equation}
\label{nonradcvprofile} \frac{1}{2\pi} \int^{2\pi}_{0} \psi_n (y,\theta)\, d\theta \to \psi^{(1)} (y), \end{equation}
as~$n$ tends to infinity and there exists an absolute constant $C$ so that
\begin{equation}
\label{nonradcdprofile} \|{\psi^{(1)}}'\|_{L^2}  \geq  C A_0\,.\end{equation}
\end{prop}
\begin{rem}
Let us point out that by virtue of the convergence properties \eqref{R4} and \eqref{R5}, the sequence ~$(\frac{1}{2\pi}  \psi_n (y,\theta))$ converges weakly  to $\psi^{(1)}(y)$
 in $L^2 (\R \times [ 0, 2\pi])$ as $n$ tends to infinity.
In particular we have for any $f \in {\mathcal D}(\R)$
$$ \frac{1}{2\pi} \int_{\R}\int^{2\pi}_{0} \psi_n (y,\theta) f(y)\, dy \, d\theta \stackrel{n\to\infty}\longrightarrow \int_{\R}\psi^{(1)} (y) f(y)\, dy.$$
In other words
\begin{equation}
\label{nonravweak}
\frac{1}{2\pi} \int_{\R}\int^{2\pi}_{0} \sqrt{\frac{2\pi}{\alpha_n^{(1)}}}\;\Big(\tau_{- x_n^{(1)}}u_n\Big) ({\rm e}^{-\alpha_n^{(1)} y} \cos \theta , {\rm e}^{-\alpha_n^{(1)} y} \sin \theta) f(y)\, dy \, d\theta \stackrel{n\to\infty}\longrightarrow \int_{\R}\psi^{(1)} (y) f(y)\, dy.\end{equation}
\end{rem}
\begin{proof}[Proof of Proposition \ref{propconvprofile}]
  Clearly $\psi^{(1)}\in {\mathcal C}(\R)$ and
${\psi^{(1)}}'=g\in L^2(\R)$. Moreover, since
$$
\Big|\psi^{(1)}(y)\Big|=\Big|\int_0^y\,g(\tau)\,d\tau\Big|\leq \sqrt{y}
\,\|g\|_{L^2(\R)},
$$
we get $\psi^{(1)}\in L^2(\R^+, {\rm e}^{-2y}\,dy)$. Now, using the fact that  $g \equiv 0$ on $ \R^- $, we get the assertion \eqref{nonradcvprofile}. Indeed,  by  virtue of \eqref{imprel}, it suffices to prove that  for any ~$y \leq 0$,
$$ F_n(y)= \frac{1}{2\pi} \int^{2\pi}_{0} \psi_n (y,\theta)\, d \theta \to 0,\quad n\to\infty\,. $$

Applying Cauchy-Schwarz's inequality and integrating with respect to $y$, we obtain
$$
\int_{-\infty}^0\,\Big|F_n(y)\Big|^2\,dy\leq \frac{1}{2\pi}\, \| \psi_n\|_{L^2 (] - \infty, 0 ] \times [ 0, 2\pi])}^2,
$$
which implies by virtue of \eqref{negcomp} that the sequence $(F_n)$ converges strongly to $0$ in $L^2 (] - \infty, 0 ])$. Therefore, up to a subsequence  extraction
$$
 F_n(y) \to 0\,\, \mbox{almost every where in}\,\,]-\infty, 0 ]\,.
 $$
Taking advantage of the continuity of ~$F_n$, we deduce that for any $y \leq 0$, $F_n(y) \to 0$  which achieves the proof of claim \eqref{nonradcvprofile}. \\

To end the proof of Proposition \ref{propconvprofile}, it remains to check \eqref{nonradcdprofile} which is the  key estimate  to iterate the process of extraction of elementary concentrations.  By virtue of  \eqref{nonradcvprofile} and Cauchy-Schwarz's inequality
$$
\Big|\psi^{(1)}(y)\Big|=\Big|\int_0^y\,{\psi^{(1)}}'(\tau)\,d\tau\Big|\leq \sqrt{y}
\,\|{\psi^{(1)}}'\|_{L^2}.
$$
So to establish the key estimate \eqref{nonradcdprofile}, it suffices in light of \eqref{nonradcvprofile} to prove the existence of $y_0$ close to $1 $  such that for $n $  big enough
\begin{equation}
\label{nonrav} \Big|\, \frac {1}{2\pi}\int^{2\pi}_{0} \psi_n (y_0,\theta)\, d\theta \,\Big| \geq  C A_0\,,\end{equation} where $C$ is an absolute constant.
For this purpose, we will again use  capacity arguments.
First, we define $$E_n^\pm  :=\Big\{x\in \R^2;\,  \pm u_n(x) \,\geq \sqrt{ 2 \alpha_n^{(1)}}\;\Big(1- \frac {\varepsilon_0} {10}\Big) \, A_0 \Big\}.   $$ Hence, $E_n = E_n^+ \cup E_n^- $. Modulo replacing  $u_n$ by $-u_n$, we can assume that \eqref{noradhearteq} yields
\begin{equation}\label{E+B}
|E^+_n \cap B (x_n^{(1)}, {\rm e}^{- (1- 2\, \varepsilon_0) \,\alpha_n^{(1)}} )| \geq  \frac{\delta_0}2  A^2_0 |E_n| \geq
 \frac{\delta_0}2  A^2_0  {\rm e}^{- 2  \,\alpha_n^{(1)}} . \end{equation}
 We also write  $u_n = u_n^+ - u_n^-$  and $\psi_n = \psi_n^+ - \psi_n^-$  where $ u_n^+,  u_n^-, \psi_n^+,  \psi_n^-  \geq 0 $.
 Introducing the set $\wt { E_n } \supset E_n^+$ defined by
$$ \wt { E_n } :=\Big\{x\in \R^2;\; u_n^+(x) \geq \frac {\sqrt{ 2 \alpha_n^{(1)}} A_0}{2}\,\Big\}, $$
we infer that we can choose $\varepsilon_0$ so that for $n$ big enough
\begin{equation}
\label{secondargcap}
\Big|\wt { E_n } \cap B (x_n^{(1)}, {\rm e}^{- (1- 2\, \varepsilon_0) \,\alpha_n^{(1)}} ) \Big| \geq \frac 1 2 \, \Big| B (x_n^{(1)}, {\rm e}^{- (1- 2\, \varepsilon_0) \,\alpha_n^{(1)}} ) \Big|\,.
\end{equation}
Otherwise taking into  account   the fact that the values of the function $u_n^+$ on  the ball $B (x_n^{(1)}, {\rm e}^{- (1- 2\, \varepsilon_0) \,\alpha_n^{(1)}} )$, that we shall designate in what follows by ${\mathbf B}$ to avoid heaviness,  varies at least from  values larger than  $\sqrt{ 2 \alpha_n^{(1)}}\;\Big(1- \frac {\varepsilon_0} {10}\Big) \, A_0$ on $E_n^+  \cap {\mathbf B}$ to  values smaller than
 $ \frac 1 2 \sqrt{2 \alpha_n^{(1)}}\, A_0$ on a subset of ${\mathbf B}$ of Lebesgue measure greater than ~$\frac { |{\mathbf B} |}{2}$, it comes from Proposition \ref{mainextraccore} that
$$
\|\nabla u_n\|_{L^2({\mathbf B})}^2  \geq  \| \nabla u_n^+\|_{L^2({\mathbf B})}^2  \geq 2 \pi \Big(\Big(\frac {1} {2} - \frac {\varepsilon_0} {10}\Big)\sqrt{ 2 \alpha_n^{(1)}}\;\, A_0\Big)^2 \frac 1 {\log \frac {{\rm e}^{-\alpha_n^{(1)}(1-2 \varepsilon_0)}}{\sqrt{|E_n^+  \cap {\mathbf B}|}}}\,. $$
Therefore, we get by virtue of Lemma \ref{noradheart} and \eqref{E+B} that
$$ \| \nabla u_n\|_{L^2({\mathbf B})}^2    \geq   \frac {C A^2_0}{ \varepsilon_0 }\,,
$$
for $n$ big enough which gives a contradiction.  Hence \eqref{secondargcap} holds.\\

Since the measure of the
 ball $B (x_n^{(1)}, {\rm e}^{- (1+ 2\, \varepsilon_0) \,\alpha_n^{(1)}} )$ is much smaller than
the measure of  ${\mathbf B}$,
we deduce  the existence of $y_0$  such that
   $1 - 2\, \varepsilon_0\leq y_0 \leq 1+ 2\, \varepsilon_0$  and~$ \psi_n^+  (y_0,\theta) \geq \sqrt{ \pi }\, A_0 $
for $\theta$ varying over an interval of length at least $\pi$.
Hence, the limit $\psi^+$ of $ \psi_n^+  $ satisfies \begin{equation}\label{condprof1}\psi^+ (y_0) \geq  \frac{ \sqrt{ \pi } }2 \, A_0\,.\end{equation}
Now, we argue in a similar way to control $\psi^- (y_0)  $ from above.
Let  $ \tilde E_n^- $  be the set  where $ u_n^- > 0  $.
The values of $u_n $ on the ball $ \mathbf B$ vary from  values larger than
$\frac {1} {2} \sqrt{ 2 \alpha_n^{(1)}} A_0$ on $E_n^+  \cap {\mathbf B} $ to negative  values
on $ \tilde E_n^-  \cap {\mathbf B} $.
Using Proposition \ref{mainextraccore}, we deduce that
$$
\|\nabla u_n\|_{L^2({\mathbf B})}^2  \geq 2\pi
 \Big(\frac {1} {2} \sqrt{ 2 \alpha_n^{(1)}}\;\, A_0\Big)^2 \frac 1 {\log \frac {\sqrt{|{\mathbf B}|}   }
{\sqrt{| \tilde E_n^- \cap  {\mathbf B}  |}}}\,, $$
which implies that $ | \tilde E_n^- \cap {\mathbf B}  | \lesssim | {\mathbf B}  |  e^{- \pi  \alpha_n^{(1)} A_0^2}   $.\\

 If $\tilde E_n^- \cap {\mathbf B} \subset B (x_n^{(1)}, {\rm e}^{- (1+ 2\, \varepsilon_0) \,\alpha_n^{(1)}} )$, we are done. If not, there exists $y_1$  such that
   $1 - 2\, \varepsilon_0\leq y_1 \leq 1+ 2\, \varepsilon_0$  and
$ \psi_n^- (y_1,\theta)   > 0  $
for $\theta$ varying over an interval of length of measure at most  $\frac{A_0} {10 M}$ and $n$ sufficiently large.
As by construction $\|\psi_n\|_{L^\infty} \leq M$, the limit $\psi^-$ of $ \psi_n^-  $ satisfies
$\psi^- (y_1) \leq  \frac{ 1 }{2\pi} \frac{  A_0  }{10}  $. Since, $\psi^-  $ belongs to the H\"older space $C^{1/2}(\R)$ and
$ |y_0 - y_1 | \leq 2 \varepsilon_0$, we deduce that
\begin{equation}\label{condprof2}\psi^- (y_0) \leq  \frac{ 1 }{2\pi} \frac{ { A_0 } }{10} + C\varepsilon_0^{1/2} \leq \frac{ { A_0 } }{10 \pi }\,,\end{equation} if
$\varepsilon_0 $ was chosen small enough. Finally combining \eqref{condprof1} and \eqref{condprof2}, we infer that
$$\psi (y_0 ) =\psi^+ (y_0)  - \psi^- (y_0)  \geq \frac{ { A_0 } }{2} $$ which achieves the proof of the last point  of Proposition \ref{propconvprofile}.
\end{proof}

\subsection{Iteration}
Our concern now is to iterate the previous process and to prove that
the algorithmic construction converges. We do not assume anymore that $u^n$ has only one sacle.
 Setting
$$ {\rm r}^{(1)}_n (x)= \sqrt{\frac{\alpha^{(1)}_n}{2\pi}}\;\left(\psi_n\Big(\frac{-\log|x -x_n^{(1)}|}{\alpha^{(1)}_n},\theta\Big) - \psi^{(1)}\Big(\frac{-\log |x -x_n^{(1)}|}{\alpha^{(1)}_n}\Big)\right),$$
where $(x_n^{(1)})$ is the sequence of points defined by \eqref{R0} and $$\psi_n (y,\theta)= \sqrt{\frac{2\pi}{\alpha_n^{(1)}}}\; \Big(\tau_{-x_n^{(1)}}u_n\Big)\Big({\rm e}^{-\alpha_n^{(1)} y} \cos \theta , {\rm e}^{-\alpha_n^{(1)} y} \sin \theta \Big)\,.$$
Let us first prove that the sequence~$({\rm r}^{(1)}_n)$ satisfies the hypothesis of Theorem \ref{noradmain}.\\

By definition~$ {\rm r}^{(1)}_n (x)= u_n(x) - g^{(1)}_n(x)$, where ~$g^{(1)}_n(x) = \sqrt{\frac{\alpha^{(1)}_n}{2\pi}}\; \psi^{(1)}(\frac{-\log|x -x_n^{(1)}|}{\alpha^{(1)}_n})$.
Noticing that
$g_n^{(1)} \rightharpoonup 0 \,\mbox{in}\,
H^1$ when  $n$ tends to infinity,   it is clear in view of \eqref{noradmain-assum1} that  $
{\rm r}^{(1)}_n$ converges weakly to $0$ as $n$ tends to infinity. \\

On the other hand, thanks to the invariance of  Lebesgue measure  under translations, we have
$$ \|\nabla u_n\|_{L^2}^2= \frac {1}{2\pi} \int_{\R}\int_0^{2\pi}\,|\partial_y\psi_n (y,\theta)|^2dy d\theta+ \frac {(\alpha_n^{(1)})^2}{2\pi}\int_{\R}\int_0^{2\pi}\,|\partial_\theta\psi_n (y,\theta)|^2dy d\theta$$
and
$$ \|\nabla g^{(1)}_n\|_{L^2}^2= \|{{\psi^{(1)}}}'\|_{L^2}^2.$$
Therefore
$$ \|\nabla {\rm r}^{(1)}_n\|_{L^2}^2= \|\nabla u_n\|_{L^2}^2 + \|{{\psi^{(1)}}}'\|_{L^2}^2 - 2 \left(\frac{1}{2\pi}\int_{\R}\int_0^{2\pi}\partial_y\psi_n (y,\theta) {{\psi^{(1)}}}'(y)\, dy d\theta \right).$$
Taking advantage of the fact that by \eqref{R5}
$$\partial_y \psi_n  \rightharpoonup {\psi^{(1)}}', \quad \mbox{as}\quad  n \to \infty \quad \mbox{in}\quad L^2(y,\theta),$$
we get
\begin{equation}
 \label{nonradnorme1}
 \lim_{n\to\infty}\,\|\nabla\,{\rm r}^{(1)}_n\|_{L^2}^2= \lim_{n\to\infty}\,\|\nabla\,u_n\|_{L^2}^2- \|{{\psi^{(1)}}}' \|_{L^2}^2\,
\end{equation}
which implies that $({\rm r}^{(1)}_n)$ is bounded in
$H^1(\R^2)$. \\

Now  since $\psi^{(1)}_{|]-\infty,0]}=0$, we obtain  for $R\geq
1$
$$
\|{\rm r}^{(1)}_n\|_{_{\cL}(|x -x_n^{(1)}|\geq R)} =\|u_n\|_{_{\cL}(|x -x_n^{(1)}|\geq R)}.$$
Taking into account of the fact that the sequence~$(x_n^{(1)})$ is bounded as it was observed in Remarks \ref{mainrems}, we deduce  that $({\rm r}^{(1)}_n)$ satisfies the hypothesis of
compactness at infinity \eqref{noradmain-assum4} and so~$({\rm r}^{(1)}_n)$ verifies the hypothesis of Theorem \ref{noradmain}. \\

 Let us then define $A_1=\limsup_{n\to\infty}\,\|{\rm r}^{(1)}_n\|_{\cL}$. If $A_1=0$, we stop the process. If not, we apply the above arguments to $ {\rm r}^{(1)}_n$  and then along the same lines as in Subsections \ref{Extractscales}, \ref{redonescale} and \ref{Extractcopro}, there exist a scale $(\alpha^{(2)}_n)$, a core  $(x^{(2)}_n)$ and a profile ~$\psi^{(2)}$ in ${\cP}$ such that
$$
{\rm
r}_n^{(1)}(x)=\sqrt{\frac{\alpha_n^{(2)}}{2\pi}}\;\psi^{(2)}\left(\frac{-\log|x - x^{(2)}_n |}{\alpha_n^{(2)}}\right)+{\rm
r}_n^{(2)}(x),
$$
with  $\|{\psi^{(2)}}'\|_{L^2}\geq C\,A_1$,  $C$ being the absolute constant appearing in
\eqref{nonradcdprofile}. This leads to the following crucial estimate
$$
\limsup_{n\to\infty}\,\|{\rm r}^{(2)}_n\|_{H^1}^2\lesssim
1-A_0^2-A_1^2.$$ Moreover, we claim that $(\alpha_n^{(1)},x_n^{(1)},\psi^{(1)})\perp(\alpha_n^{(2)}, x_n^{(2)},\psi^{(2)})$ in the sense of the terminology introduced in Definition \ref{orthogen}. In fact, if  $\alpha^{(1)}_n\perp \alpha^{(2)}_n$, we are done. If not, by virtue of \eqref{inara}, we can suppose that $\alpha_n^{(1)} = \alpha_n^{(2)} = \alpha_n$ and  our purpose is then to prove that
\begin{equation} \label{ortoproof} - \frac{\log|x_n^{(1)}-  x_n^{(2)}|}{\alpha_n} \longrightarrow a\,\, \mbox{as}\,\, n \to \infty,\,\,\mbox{with}\,\psi^{(1)} \,\, \mbox{or}\,\,\psi^{(2)}  \,\, \mbox{null for}\,\, s <a\,.
\end{equation}
First of all assuming that $- \frac{\log|x_n^{(1)}-  x_n^{(2)}|}{\alpha_n} \to a$, let us  prove  that the profile $\psi^{(2)}$ is null for $s< a$. For this purpose, let us begin by recalling that  in view of \eqref{nonravweak}  we have  for any
$f \in {\mathcal D}(\R)$,
$$\sqrt{\frac{1}{2\pi\alpha_n}} \int_{\R}\int^{2\pi}_{0} {\rm
r}^{(1)}_n \Big(x_n^{(2)} + ({\rm e}^{- \alpha_n \,t} \cos \lambda, {\rm e}^{-\alpha_n t} \sin \lambda)\Big) f(t) dt \, d\lambda \to \int_{\R}\psi^{(2)} (t) f(t)\, dt,$$
as $n$ tends to infinity.  \\

Consequently for  fixed $\varepsilon > 0$  we are reduced to demonstrate that for any function $f \in {\mathcal D}(] -\infty, a - \varepsilon[)$, we have
\begin{equation}\label{convint3}
\lim_{n \to \infty}\sqrt{\frac{1}{2\pi\alpha_n}} \int^{a - \varepsilon}_{-\infty}\int^{2\pi}_{0} {\rm
r}^{(1)}_n \Big(x_n^{(2)} + ({\rm e}^{- \alpha_n \,t} \cos \lambda, {\rm e}^{- \alpha_n \,t} \sin \lambda)\Big) f(t)\, dt \, d\lambda = 0.\end{equation}
To do so, let us perform the change of variables
\begin{equation}\label{change} ({\rm e}^{-\alpha_ns} \cos \theta, {\rm e}^{-\alpha_ns} \sin \theta) = x_n^{(2)}-x_n^{(1)} + ({\rm e}^{-\alpha_n t_n(s,\theta)} \cos \lambda_n(s,\theta), {\rm e}^{-\alpha_n t_n(s,\theta)} \sin \lambda_n(s,\theta)).\end{equation}
Denoting by $J_n(s,\theta)$  the Jacobian of the change of variables \eqref{change}, we claim that
\begin{equation}\label{change2}  t_n(s,\theta) \stackrel{n\to\infty}\longrightarrow s,
\end{equation}
and
\begin{equation}\label{change5}  J_n(s,\theta)\stackrel{n\to\infty}\longrightarrow 1,
\end{equation}
uniformly with respect to $(s,\theta) \in ]-\infty, a- \frac \varepsilon 2]\times [0,2\pi]$.\\

Indeed, firstly we can  observe that
\begin{equation}\label{change6}  {\rm e}^{-\alpha_n t_n(s,\theta)} = {\rm e}^{ - \alpha_n s} | 1 + \Theta_n (s,\theta)| = {\rm e}^{ - \alpha_n s} ( 1 + o(1)), \end{equation}
where $\Theta_n (s,\theta) = {\rm e}^{\alpha_n s} {\rm e}^{-i\theta} z_n$, with $z_n$  the writing in $\C$ of the point $x_n^{(1)}-x_n^{(2)}$.
But by hypothesis, we know that  $- \frac{\log|x_n^{(1)}-  x_n^{(2)}|}{\alpha_n} \stackrel{n\to\infty}\longrightarrow  a$. Then, there exists an integer $N$ such that for all $n > N$,
\begin{equation}\label{change3} {\rm e}^{- \alpha_n (a+\frac \varepsilon 4 )} \leq |x_n^{(1)}-  x_n^{(2)}| \leq {\rm e}^{- \alpha_n (a-\frac \varepsilon 4 )}\end{equation}
which according to \eqref{change} and the fact that $t \in ]-\infty,a- \varepsilon]$ gives rise (for $n > N$) to
$${\rm e}^{- \alpha_n \,t_n(s,\theta)} (1-{\rm e}^{-  \frac {\alpha_n \varepsilon}{ 4} })\leq  {\rm e}^{- \alpha_n \,s} \leq  {\rm e}^{- \alpha_n \,t_n(s,\theta)} (1+{\rm e}^{-  \frac {\alpha_n \varepsilon}{ 4} }).$$
This implies that  $s \leq a - \frac {\varepsilon}{ 2}$ for $n$ big enough, and leads to \eqref{change6}. We deduce that
\begin{equation}\label{change7}  t_n(s,\theta) = s - \frac {\log| 1 + \Theta_n (s,\theta)|}{\alpha_n},
\end{equation}
with $|\Theta_n (s,\theta)| \leq {\rm e}^{- \frac {\alpha_n \varepsilon}{ 4} }$ for $s \leq a - \frac {\varepsilon}{ 2}$ and $n$ sufficiently large, which achieves the proof of \eqref{change2}.
Now, since $\frac {\partial \Theta_n} {\partial s} = \alpha_n \Theta_n$ and $\frac {\partial \Theta_n} {\partial \theta} = -i \Theta_n$, it follows that
$$\Big|\frac {\partial \Theta_n} {\partial s}\Big|\leq \alpha_n {\rm e}^{- \frac {\alpha_n \varepsilon}{ 4} } \quad \mbox{and} \quad \Big|\frac {\partial \Theta_n} {\partial \theta}\Big|\leq  {\rm e}^{- \frac {\alpha_n \varepsilon}{ 4} },$$ which easily implies that
\begin{equation}\label{jac1} \frac {\partial t_n} {\partial s} = 1 + o(1) \quad \mbox{and} \quad \frac {\partial t_n} {\partial \theta} = o(1),\end{equation}
where $|o(1)| \lesssim {\rm e}^{- \frac {\alpha_n \varepsilon}{ 4} }$  for $s \leq a - \frac {\varepsilon}{ 2}$ and $n$ big enough.

Otherwise, taking account of \eqref{change} and  \eqref{change7}, we get
$$ {\rm e}^{i \, \lambda_n(s,\theta)} =   \frac {( 1 + \Theta_n (s,\theta))}{| 1 + \Theta_n (s,\theta)|} \, {\rm e}^{i \, \theta}, $$ which allows by straightforward computations to prove that
\begin{equation}\label{jac2} \frac {\partial \lambda_n} {\partial s} =  o(1) \quad \mbox{and} \quad \frac {\partial \lambda_n} {\partial \theta} = 1+ o(1),\end{equation}
again with $|o(1)| \lesssim {\rm e}^{- \frac {\alpha_n \varepsilon}{ 4} }$  for $s \leq a - \frac {\varepsilon}{ 2}$ and $n$ big enough. Finally, the combination of \eqref{jac1} and \eqref{jac2} ensures claim \eqref{change5}.\\

Now,  making  the change of variables \eqref{change}, the  left hand side of \eqref{convint3} becomes for n big enough
$$ I_n:= \sqrt{\frac{1}{2\pi\alpha_n}} \int^{a - \frac{\varepsilon}{2}}_{-\infty}\int^{2\pi}_{0} {\rm
r}^{(1)}_n \Big(x_n^{(1)} + ({\rm e}^{- \alpha_n \,s} \cos \theta, {\rm e}^{-  \alpha_n \,s} \sin \theta)\Big) f(t_n(s,\theta)) J_n(s,\theta)\, ds \, d\theta.$$
But, by definition
$$  {\rm
r}^{(1)}_n(x) = u_n(x) - \sqrt{\frac{\alpha_n}{2\pi}}\; \psi^{(1)}\left(\frac{-\log|x -x_n^{(1)}|}{\alpha_n}\right).$$
Thus
$$ I_n= \frac{1}{2\pi}\int^{a - \frac{\varepsilon}{2}}_{-\infty}\int^{2\pi}_{0}\Big( \sqrt{\frac{2\pi}{\alpha_n}}u_n \big(x_n^{(1)} + ({\rm e}^{- \alpha_n \,s} \cos \theta, {\rm e}^{-  \alpha_n \,s} \sin \theta)\big)-\psi^{(1)}(s)\Big) f(t_n(s,\theta)) J_n(s,\theta) ds d\theta.$$

Remembering that in view of \eqref{nonravweak}  we have   for any
$f \in {\mathcal D}(\R)$
 $$ \sqrt{\frac{1}{2\pi\alpha_n}}\int_{\R}\int^{2\pi}_{0} u_n\big(x_n^{(1)} + ({\rm e}^{-\alpha_n s} \cos \theta, {\rm e}^{-\alpha_n s} \sin \theta)\big) f(s) ds  d\theta \stackrel{n\to\infty}\longrightarrow \int^\infty_{0}\psi^{(1)}(s) f(s) ds,$$
 it comes under the fact that $t_n(s,\theta) \stackrel{n\to\infty}\longrightarrow s \,\, \mbox{and}\,\, J_n(s,\theta)\stackrel{n\to\infty}\longrightarrow 1$ uniformly with respect to $(s,\theta)$ that
$$ I_n \stackrel{n\to\infty}\longrightarrow \frac{1}{2\pi}\int^{a - \frac{\varepsilon}{2}}_{0}\int^{2\pi}_{0}\Big(\psi^{(1)}(s) - \psi^{(1)}(s)\Big) f(s)\, ds \, d\theta = 0,$$which concludes the proof of the fact that $\psi^{(2)}$ is null for $s< a$. \\

To achieve the proof of \eqref{ortoproof}, it remains to show  that the sequence $$t_n := \frac{\log|x_n^{(1)}-  x_n^{(2)}|}{\alpha_n}$$ is bounded and so, up to a subsequence
extraction, it converges. Indeed, on one hand  by virtue of Remark \ref{mainrems} the sequences of cores $(x_n^{(1)})$ and $(x_n^{(2)})$ are bounded which ensures the existence of a positive constant $M$ such that $t_n \leq \frac{\log M }{\alpha_n}$. On the other hand,  there is $\gamma > 0$ such that
\begin{equation}\label{mincores}|x_n^{(1)}-  x_n^{(2)}| \geq  {\rm e}^{- \gamma \,\alpha_n}, \end{equation} for $n$ big enough. In effect,  by construction $x_n^{(2)}$ is chosen so that  there exist an absolute constant  $C$ and $t_0> 0$ such that
$$ \frac{1}{ 2\pi}\left| \int^{2\pi}_{0} \sqrt{\frac{2\pi}{\alpha_n}}\;\Big(\tau_{- x_n^{(2)}}{\rm
r}_n^{(1)}\Big) ({\rm e}^{-\alpha_n t_0} \cos \lambda , {\rm e}^{-\alpha_n t_0} \sin \lambda)  \, d\lambda \right| \geq  C A_1\,,$$   for $n $  sufficiently large, which can be written   by designating  $x_n^{(1)}- x_n^{(2)}$ by $w_n$ $$ \frac{1}{ 2\pi}\left| \int^{2\pi}_{0} \Big(\psi_n\Big(\frac{-\log|{\rm e}^{-\alpha_n t_0}{\rm e}^{{\rm i}\lambda} - w_n |}{\alpha_n},\lambda\Big) - \psi^{(1)}\Big(\frac{-\log |{\rm e}^{-\alpha_n t_0} {\rm e}^{{\rm i}\lambda} -w_n |}{\alpha_n}\Big)\Big)  \, d\lambda \right|| \geq  C A_1\,,$$ and implies that $| \psi^{(2)}(t_0)| \geq C A_1\,$.
Since $\psi^{(2)}$ is continuous, there exists $\delta > 0$ such that for $t \in [t_0- \delta, t_0 + \delta]$, we have
$$| \psi^{(2)}(t)| \geq \frac{C}{ 2} A_1\,.$$
According to \eqref{nonravweak}, we deduce     that if $f $ is a positive valued function  belonging to $ {\mathcal D}([t_0- \delta, t_0 + \delta])$ and identically equal to one in $[t_0- \frac{\delta}{ 2}, t_0 + \frac{\delta}{ 2}]$,  then  for $n$ large enough
\begin{equation}\label{mincores2}\left| \sqrt{\frac{1}{2\pi\alpha_n}} \int^{t_0+ \delta}_{t_0- \delta}\int^{2\pi}_{0} {\rm
r}^{(1)}_n \Big(x_n^{(2)} + ({\rm e}^{- \alpha_n \,t} \cos \lambda, {\rm e}^{-\alpha_n t} \sin \lambda)\Big) f(t) dt \, d\lambda \right| \geq  \frac{C \, \delta}{ 2} A_1\,.\end{equation}
 Now if claim \eqref{mincores} does not hold,    then ${\rm e}^{-\alpha_n (t_0 - 2\delta)} w_n \to 0$. This yields along the same lines as above  making use of  the change of variables   \eqref{change} to
$$ \sqrt{\frac{1}{2\pi\alpha_n}} \int^{t_0+ \delta}_{t_0- \delta}\int^{2\pi}_{0} {\rm
r}^{(1)}_n \Big(x_n^{(2)} + ({\rm e}^{- \alpha_n \,t} \cos \lambda, {\rm e}^{-\alpha_n t} \sin \lambda)\Big) f(t) dt \, d\lambda \to 0\,,$$
which contradicts \eqref{mincores2}.
Thus, there exists an integer $N_1$ such that for all $n \geq N_1$
\begin{equation} \label{ortoestcor}  - \gamma \leq  t_n  \leq \frac{\log M }{\alpha_n}\end{equation}
and then if $\ell$ denotes the limit of the sequence $(t_n)$ (up to a subsequence
extraction), we have   necessary $\ell = - a $ with $0 \leq a \leq \gamma $. This ends  the proof of the orthogonality property \eqref{ortoproof}.\\

Finally, by iterating the process we get at step $\ell$
$$
u_n(x)=\Sum_{j=1}^{\ell}\,\sqrt{\frac{\alpha_n^{(j)}}{2\pi}}\;\psi^{(j)}\left(\frac{-\log|x -x_n^{(j)}|}{\alpha_n^{(j)}}\right)+{\rm
r}_n^{(\ell)}(x),
$$
with
$$
\limsup_{n\to\infty}\,\|{\rm r}^{(\ell)}_n\|_{H^1}^2\lesssim
1-A_0^2-A_1^2-\cdots -A_{\ell-1}^2\,.
$$
Therefore $A_\ell\to 0$ as $\ell\to\infty$ which achieves the proof of the asymptotic decomposition \eqref{noraddecomp}.\\


\subsection{Proof of the orthogonality equality}
To end the proof of Theorem \ref{noradmain}, it remains to establish the orthogonality equality \eqref{ortogonal2}. Since in an abstract Hilbert space $H$, we have
$$
\|\sum_{j=1}^\ell h_j\|_H^2=\sum_{j=1}^\ell\|h_j\|_H^2+\sum_{j\neq k}\left(h_j,h_k\right)_H,
$$
we shall restrict ourselves to two elementary concentrations, and prove the following result.
\begin{prop}
\label{stabi}
Let us consider
\beqn
f_n(x)&=&\sqrt{\frac{\alpha_n}{2\pi}}\,\psi\left(\frac{-\log|x-x_n|}{\alpha_n}\right) \quad \mbox{and}\\
g_n(x)&=&\sqrt{\frac{\beta_n}{2\pi}}\,\varphi\left(\frac{-\log|x-y_n|}{\beta_n}\right),
\eeqn
where ${\underline\alpha}$, ${\underline\beta}$ are two scales, ${\underline x}$, ${\underline y}$ are two cores and $\varphi$, $\psi$ are two profiles such that $({\underline\alpha}, {\underline x}, \psi)\perp({\underline\beta}, {\underline y}, \varphi)$ in the sense of Definition \ref{orthogen}.
Then
\bq
\label{stab}
\|\nabla f_n+\nabla g_n\|_{L^2}^2=\|\nabla f_n\|_{L^2}^2+\|\nabla g_n\|_{L^2}^2+\circ(1),\;\; \mbox{as} \;\; n\to\infty.
\eq
\end{prop}
\begin{proof}
To prove \eqref{stab} it suffices to show that the sequence $I_n:=\left(\nabla f_n, \nabla g_n\right)_{L^2}$ tends to zero as $n$ goes to infinity. Using density argument (see Lemma \ref{stabicompact}) and the space translation invariance, we can assume without loss of generality that $\underline y=0$ and $\varphi, \psi \in {\cD}$ . Write
\beqn
I_n&=&\frac{1}{2\pi\sqrt{\alpha_n\beta_n}}\,\int_{\R^2}\,\psi'\left(\frac{-\log|x-x_n|}{\alpha_n}\right)
\varphi'\left(\frac{-\log|x|}{\beta_n}\right)\frac{x}{|x|^2}\cdot\frac{x-x_n}{|x-x_n|^2}\,dx\\\\
&=&\frac{1}{2\pi\sqrt{\alpha_n\beta_n}}\,\int_{\R^2}\,\psi'\left(\frac{-\log|x|}{\alpha_n}\right)
\varphi'\left(\frac{-\log|x+x_n|}{\beta_n}\right)\frac{x}{|x|^2}\cdot\frac{x+x_n}{|x+x_n|^2}\,dx.
\eeqn
The change of variable $|x|=r={\rm e}^{-\alpha_n s}$ yields  to
\bq
\label{In}
I_n=\frac{1}{2\pi}\sqrt{\frac{\alpha_n}{\beta_n}}\,\int_0^\infty\int_0^{2\pi}\,\psi'(s)\,\varphi'\left(\frac{-\log|{\rm e}^{-\alpha_n s}{\rm e}^{{\rm i}\theta}+x_n|}{\beta_n}\right)\,{\rm e}^{{\rm i}\theta}\cdot\frac{{\rm e}^{{\rm i}\theta}+{\rm e}^{\alpha_n s}x_n}{|{\rm e}^{{\rm i}\theta}+{\rm e}^{\alpha_n s}x_n|^2}\,ds\,d\theta.
\eq

According to the definition of the orthogonality, we shall distinguish two cases. Let us assume first that $\underline\alpha=\underline\beta$, $-\frac{\log|x_n|}{\alpha_n}\to a\geq 0$ and that $\psi$ is null for $s\leq a$.  Taking advantage of the fact that $\psi'\in L^2([a,\infty[)$,  we infer that for any $\eps>0$ there exists $a<b<B$ such that
$$
\|\psi'\|_{L^2(a\leq s\leq b)}\leq\eps\;\;\;\mbox{and}\;\;\;\|\psi'\|_{L^2( s\geq B)}\leq\eps\,.
$$
We decompose then $I_n$ as follows
$$
I_n=J_n+K_n:=\int_{\{{\rm e}^{-B\alpha_n}\leq |x|\leq {\rm e}^{- b\alpha_n}\}}\,\nabla f_n(x)\cdot\nabla g_n(x)\,dx+K_n\,.
$$
Clearly, $|K_n|\lesssim\eps$ and
$$
J_n=\frac{1}{2\pi}\,\int_b^B\int_0^{2\pi}\,\psi'(s)\,\varphi'\left(\frac{-\log|{\rm e}^{-\alpha_n s}{\rm e}^{{\rm i}\theta}+x_n|}{\alpha_n}\right)\,{\rm e}^{{\rm i}\theta}\cdot\frac{{\rm e}^{{\rm i}\theta}+{\rm e}^{\alpha_n s}x_n}{|{\rm e}^{{\rm i}\theta}+{\rm e}^{\alpha_n s}x_n|^2}\,ds\,d\theta\,.
$$
Now, since $-\frac{\log|x_n|}{\alpha_n}\stackrel{n\to\infty}\longrightarrow a\geq 0\,$, we get for all $s>a$,
$${\rm e}^{s\alpha_n}\,|x_n|\stackrel{n\to\infty}\longrightarrow\infty.$$
Likewise for any $b>a$ and $n$ large enough,
$$
{\rm e}^{s\alpha_n}\,|x_n|-1\geq \frac{1}{2}{\rm e}^{s\alpha_n}\,|x_n|\quad\mbox{uniformly in}\quad s\in[b,\infty[\,.
$$
Hence
\beqn
|J_n|&\lesssim&\int_b^B\int_0^{2\pi}\,|\psi'(s)|\,\Big|\varphi'\left(\frac{-\log|{\rm e}^{-\alpha_n s}{\rm e}^{{\rm i}\theta}+x_n|}{\alpha_n}\right)\Big|\frac{ds\,d\theta}{{\rm e}^{s\alpha_n}|x_n|}\\
&\lesssim&\frac{1}{{\rm e}^{b\alpha_n}|x_n|}\|\psi'\|_{L^\infty}\|\varphi'\|_{L^\infty}\to 0\,.
\eeqn

The second case $\underline\alpha\perp\underline\beta$ can be handled in a similar way. Indeed, assuming that $\underline\alpha\perp\underline\beta$ and $-\frac{\log|x_n|}{\alpha_n}\to a\geq 0$, we get (for $\eta>0$ small enough and $n$ large),
\bq
\label{eta}
{\rm e}^{-(a+\eta)\alpha_n}\leq |x_n|\leq {\rm e}^{-(a-\eta)\alpha_n}\,.
\eq
If $a=0$, we argue exactly as above to obtain
$$
|J_n|\lesssim\sqrt{\frac{\alpha_n}{\beta_n}}\,\frac{1}{{\rm e}^{(b-\eta)\alpha_n}}\|\psi'\|_{L^\infty}\|\varphi'\|_{L^\infty}\to 0\,.
$$
If $a>0$, we decompose $J_n$ as follows
\beqn
J_n&=&\frac{1}{2\pi}\sqrt{\frac{\alpha_n}{\beta_n}}\,\int_b^{a-\delta}\int_0^{2\pi}\,\psi'(s)\,\varphi'\left(\frac{-\log|{\rm e}^{-\alpha_n s}{\rm e}^{{\rm i}\theta}+x_n|}{\beta_n}\right)\,{\rm e}^{{\rm i}\theta}\cdot\frac{{\rm e}^{{\rm i}\theta}+{\rm e}^{\alpha_n s}x_n}{|{\rm e}^{{\rm i}\theta}+{\rm e}^{\alpha_n s}x_n|^2}\,ds\,d\theta\\&+&\frac{1}{2\pi}\sqrt{\frac{\alpha_n}{\beta_n}}\,\int_{a-\delta}^{a+\delta}\int_0^{2\pi}\,\psi'(s)\,\varphi'\left(\frac{-\log|{\rm e}^{-\alpha_n s}{\rm e}^{{\rm i}\theta}+x_n|}{\beta_n}\right)\,{\rm e}^{{\rm i}\theta}\cdot\frac{{\rm e}^{{\rm i}\theta}+{\rm e}^{\alpha_n s}x_n}{|{\rm e}^{{\rm i}\theta}+{\rm e}^{\alpha_n s}x_n|^2}\,ds\,d\theta\\&+&\frac{1}{2\pi}\sqrt{\frac{\alpha_n}{\beta_n}}\,\int_{a+\delta}^B\int_0^{2\pi}\,\psi'(s)\,\varphi'\left(\frac{-\log|{\rm e}^{-\alpha_n s}{\rm e}^{{\rm i}\theta}+x_n|}{\beta_n}\right)\,{\rm e}^{{\rm i}\theta}\cdot\frac{{\rm e}^{{\rm i}\theta}+{\rm e}^{\alpha_n s}x_n}{|{\rm e}^{{\rm i}\theta}+{\rm e}^{\alpha_n s}x_n|^2}\,ds\,d\theta\\&=&J_n^1+J_n^2+J_n^3,
\eeqn
where $\delta>0$ is a small parameter to be chosen later. Clearly
$$
|J_n^3|\lesssim \sqrt{\frac{\alpha_n}{\beta_n}}\,\frac{1}{{\rm e}^{(\delta-\eta)\alpha_n}}\to 0,
$$
provided that $\eta<\delta$. Besides, by Cauchy-Schwarz inequality (as for the term $K_n$), we have $|J_n^2|\leq \eps$ for $\delta$ small enough. It remains to deal with the first term $J_n^1$. Using estimate \eqref{eta}, we get for $\eta<\delta$ and $b\leq s\leq a-\delta$,
$$
|{\rm e}^{s\alpha_n}\,x_n|\leq {\rm e}^{(\eta-\delta)\alpha_n}\to 0.
$$
 Therefore
$$
 |J_n^1|\lesssim \sqrt{\frac{\alpha_n}{\beta_n}}\,\int_b^{a-\delta}\int_0^{2\pi}\,|\psi'(s)|\left|\varphi'\left(\frac{-\log|{\rm e}^{-\alpha_n s}{\rm e}^{{\rm i}\theta}+x_n|}{\beta_n}\right)\right|\,ds\,d\theta\,.
 $$
If $\frac{\alpha_n}{\beta_n}\to 0$, we are done thanks to the estimate
 $$
|J^1_n|\lesssim\sqrt{\frac{\alpha_n}{\beta_n}}\,\|\psi'\|_{L^\infty}\|\varphi'\|_{L^\infty}.
$$
If not,  the integral $J_n^1$  is null for n large enough which ensures the result. Indeed, in view of \eqref{eta} and the fact $\frac{\alpha_n}{\beta_n}\to \infty$,  we have for $b\leq s\leq a-\delta$
 $$
 \frac{-\log|{\rm e}^{-\alpha_n s}{\rm e}^{{\rm i}\theta}+x_n|}{\beta_n}\geq b \,\frac{\alpha_n}{\beta_n}+\circ(1)\, \to \infty.
 $$
Since $\varphi \in {\cD}$,  we deduce  that  for $n$ large enough  $\varphi'\left(\frac{-\log|{\rm e}^{-\alpha_n s}{\rm e}^{{\rm i}\theta}+x_n|}{\beta_n}\right)=0$, and thus $J_n^1$  is null.
The proof of Proposition \ref{stabi} is then completely achieved.
\end{proof}

\section{Appendix}\label{appendix}


\subsection{Appendix A: Comments on elementary concentrations}
\label{examples}

In this appendix, we will present some more general examples than the one illustrated in Lemma \ref{anisoexamp}. But first recall the following relevant result proved in \cite{Orlicz-Book}.
\begin{lem}
\label{monoto}
 We have the following properties\\
\noindent$\bullet$ {\it Lower semi-continuity}:
$$u_n\to u\quad\mbox{\sf a.e.}\quad\Longrightarrow\quad\|u\|_{\cL}\leq
\Liminf_{n\to\infty}\|u_n\|_{\cL}.
$$
\noindent$\bullet$ {\it Monotonicity}:
$$
|u_1|\leq|u_2|\quad\mbox{\sf
a.e.}\quad\Longrightarrow\quad\|u_1\|_{\cL}\leq\|u_2\|_{\cL}.
$$
\noindent$\bullet$ {\it Strong Fatou property}:
$$
0\leq u_n\nearrow u\quad\mbox{\sf
a.e.}\quad\Longrightarrow\quad\|u_n\|_{\cL}\nearrow\|u\|_{\cL}.
$$
\end{lem}

The first example is of type $f_{\alpha_n}\circ \varphi$. More precisely, we have the following result.
\begin{prop}
\label{genrad} Let $\psi\in{\cP}$, $(\alpha_n)$ a scale
and set
$$
g_n(x):=\sqrt{\frac{\alpha_n}{2\pi}}\;\psi\Big(\frac{-\log|\varphi(x)|}{\alpha_n}\Big),
$$
where~$\varphi : \R^2 \longrightarrow \R^2 $ is a global diffeomorphism  vanishing at the origin
 and satisfying ~
\bq
\label{phi}
 \det\left(\dd\varphi^{-1}\right)\in L^\infty\quad\mbox{and}\quad \|\dd\,\varphi\| \in L^\infty\,.
\eq
Then
 $$ g_n(x) \asymp \sqrt{\frac{\alpha_n}{2\pi}}\;\psi\Big(\frac{-\log|x|}{\alpha_n}\Big) \quad\mbox{in}\quad \cL\,.$$
\end{prop}
\begin{rem}
Assumption \eqref{phi} is required to ensure that the sequence $(g_n)$  is bounded  in $ H^1$,  converges strongly to $0$ in $L^2$ and satisfies
 $\Liminf_{n\to\infty}\,\|g_n\|_{{\cL}}>0$.
\end{rem}
\begin{proof}[Proof of Proposition \ref{genrad}]
Assumption \eqref{phi}  ensures that $\|\dd\,\varphi^{-1}\| \in L^\infty$. Hence, one can write
\bq
\label{phiencad}
\frac{|x|}{\|\dd\,\varphi^{-1}\|_{L^\infty}}\leq |\varphi(x)|\leq \|\dd\,\varphi\|_{L^\infty}\,|x|\,.
\eq
 Let us first assume that the profile $\psi$ is increasing and set
$$
w_n(x)=g_n(x)-\sqrt{\frac{\alpha_n}{2\pi}}\;\psi\Big(\frac{-\log|x|}{\alpha_n}\Big).
$$
Therefore, we get
$$
w_n^1(x)\leq w_n(x)\leq w_n^2(x),
$$
where
\beqn
w_n^1(x)&=&\sqrt{\frac{\alpha_n}{2\pi}}\;\psi\Big(\frac{-\log \lam_1|x|}{\alpha_n}\Big)-\sqrt{\frac{\alpha_n}{2\pi}}\;\psi\Big(\frac{-\log|x|}{\alpha_n}\Big),\\
w_n^2(x)&=&\sqrt{\frac{\alpha_n}{2\pi}}\;\psi\Big(\frac{-\log \lam_2|x|}{\alpha_n}\Big)-\sqrt{\frac{\alpha_n}{2\pi}}\;\psi\Big(\frac{-\log|x|}{\alpha_n}\Big)\,,
\eeqn
with~$\lam_1= \|\dd\,\varphi\|_{L^\infty}$ and ~$\lam_2 = \frac{1}{\|\dd\,\varphi^{-1}\|_{L^\infty}}$.
We deduce that  for all $x \in \R^2$
$$ \big| w_n(x) \big| \leq \big| w_n^1(x) \big| + \big| w_n^2(x) \big|
\leq \Big|\big| w_n^1(x) \big| + \big| w_n^2(x) \big| \Big| $$
which ensures the result in the same fashion than in Lemma \ref{anisoexamp}. \\

To achieve the proof of the proposition, it suffices to remark that any profile $\psi$ can be decomposed as follows
\beqn
\psi(y)&=&\frac{1}{2}\,\int_0^y\,\Big(|\psi'(\tau)|+\psi'(\tau)\Big)\,d\tau-\frac{1}{2}\,\int_0^y\,\Big(|\psi'(\tau)|-\psi'(\tau)\Big)\,d\tau\\
&:=&\psi_1(y)-\psi_2(y)\,.
\eeqn
Since $\psi_1$ and $\psi_2$ belong to the set of profiles $\cP$ and are increasing, the result follows from the first step.
\end{proof}

By taking some particular sequences of diffeomorphims, we derive other examples as follows.
\begin{lem}
\label{examp2}
 Let us define the sequence $(u_\alpha)$ by
$$
u_\alpha(x)=f_\alpha\left(\bA_\alpha x\right),
$$ where $(\bA_\alpha)$ is  a sequence of $2\times 2$ invertible matrix satisfying
\bq
\label{matrix1}
\alpha|\det\,\bA_\alpha|\longrightarrow\infty,
\eq
\bq
\label{matrix2}
\|\bA_\alpha\|^2\lesssim |\det\bA_\alpha| \quad\mbox{and}
\eq
\bq
\label{matrix3}
\frac{\log|\det\bA_\alpha|^{1/2}}{\alpha}\longrightarrow\,a\geq 0,\;\;\;\alpha\to\infty.
\eq
Then
$$
u_\al(x)\,\asymp\,\sqrt{\frac{\alpha}{2\pi}}\;{\mathbf L}_{a}\,\Big(\frac{-\log|x|}{\alpha}\Big)\quad\mbox{in}\quad\,\cL\,,
$$
with ${\mathbf L}_{a}(s) := {\mathbf L} (s-a)$.
\end{lem}

\begin{proof}[Proof of Lemma \ref{examp2}]

Note that assumption \eqref{matrix1} implies that the sequence $(u_\alpha)$ converges strongly to $0$ in $L^2$, while the second assumption \eqref{matrix2} tell us that $(\nabla\,u_\alpha)$ is bounded in $L^2$. It follows that $(u_\alpha)$ converges weakly to $0$ in $H^1$. To show that the sequence $(u_\alpha)$ does not tend to $0$ in the Orlicz space $\cL$, we first remark that thanks to the simple inequality $|\bA_\alpha\,x|\leq \|\bA_\alpha\||x|$, we have
$$
u_\alpha(x)=\sqrt{\frac{\alpha}{2\pi}}\quad\mbox{if}\quad |x|\leq \frac{{\rm e}^{-\alpha}}{\|\bA_\alpha\|}\,.
$$
Therefore, if
$$
\Int_{\R^2}\,\left({\rm e}^{\frac{|u_\al(x)|^2}{\lam^2}}-1\right)\,dx\leq\kappa,
$$
 then
$$
2\pi\Int_0^{\frac{{\rm e}^{-\alpha}}{\|\bA_\alpha\|}}\,\left({\rm e}^{\frac{\al}{2\pi\lam^2}}-1\right)\,r\,dr\leq\kappa\,.
$$
Taking advantage of \eqref{matrix2}, we deduce  that
$$
\lam^2\geq \frac{\alpha}{2\pi\log\left(1+C\,{\rm e}^{2\al}\,|\det\bA_\al|\right)},
$$
for some absolute positive constant $C$. This leads by means of \eqref{matrix3} to
$$
\Liminf_{\alpha\to\infty}\,\|u_\al\|_{\cL}\geq \frac{1}{2 \sqrt{\pi(1+a)} }\,.
$$
Our aim now is to prove that the sequence $(u_\alpha)$ behaves like the  sequence
$$
\sqrt{\frac{\alpha}{2\pi}}\;{\mathbf L}_{a}\,\Big(\frac{-\log|x|}{\alpha}\Big)
$$
 in the sense that the difference goes to $0$ in  the Orlicz space $\cL$ as $\alpha$ tends to infinity. For this purpose, we shall use the following elementary result from linear algebra.
\begin{lem}
\label{LinAlg}
Let $\bA$ be a $2\times2$ invertible matrix. Then
\bq
\label{LinAl}
\|\bA^{-1}\|=\frac{\|\bA\|}{|\det\bA|}\,.
\eq
\end{lem}
Now combining  Lemma \ref{LinAlg} together with assumption \eqref{matrix2}, we infer that
$$
\frac{1}{C}|\det\bA_\al|^{1/2}\,|x|\leq \frac{|\det\bA_\al|}{\|\bA_\al\|}\,|x|\leq |\bA_\al\,x|\leq \|\bA_\al\||x|\leq\,C|\det\bA_\al|^{1/2}\,|x|\,.
$$
Reasoning exactly as in Proposition \ref{genrad}, it comes in light of  \eqref{matrix1} that
$$
\|u_\al-v_\al\|_{\cL}\longrightarrow 0,\quad \al\to\infty,
$$
where
$$
v_\al(x)=\sqrt{\frac{\alpha}{2\pi}}\;{\mathbf L}\,\left(\frac{-\log|a_\alpha\,x|}{\alpha}\right)\;\;\mbox{with}\;\;a_\al=|\det\bA_\al|^{1/2}\,,
$$
which leads to the result.
\end{proof}


\subsection{Appendix B: Proof of Proposition \ref{sumOrlicz2}}
\label{proofprop{sumOrlicz2}}
The proof goes the same lines as the proof of Proposition 1.18 in \cite{BMM}, but  before entering into the details, let us show the relevance of the orthogonality assumption between the scales. For that purpose, we shall treat a simple example.
\begin{lem}
\label{cex}
Let $(\alpha_n)$ be a scale, and set
\beqn
u_n(x)&=&\sqrt{\frac{\alpha_n}{2\pi}}\,{\mathbf L}\left(-\frac{\log|x|}{\alpha_n}\right)+\sqrt{\frac{\alpha_n}{2\pi}}\,{\mathbf L}_1\left(-\frac{\log|x-x_n|}{\alpha_n}\right)\\
&=&f_n(x)+g_n(x),
\eeqn
where ${\mathbf L}_1(s)={\mathbf L}(s-1)$, and $|x_n|=\frac{1}{2}\,{\rm e}^{-\alpha_n}$. Then, we have
\bq
\label{cex1}
\liminf_{n\to\infty}\|u_n\|_{\cL}\;>\frac{1}{\sqrt{4\pi}}=\Max\Big(\lim_{n\to\infty}\|f_n\|_{\cL},\;\lim_{n\to\infty}\|g_n\|_{\cL}\Big)\,.
\eq
\end{lem}
\begin{proof}[Proof of Lemma \ref{cex}]
Recall that, by definition of ${\mathbf L}$, we have
\begin{eqnarray*}
f_n(x)&=&\; \left\{
\begin{array}{cllll}\sqrt{\frac{\alpha_n}{2\pi}} \quad&\mbox{if}&\quad
|x|\leq{\rm e}^{-\alpha_n},\\\\-\frac{\log|x|}{\sqrt{2\pi\alpha_n}}\quad
&\mbox{if}&\quad {\rm e}^{-\alpha_n}\leq |x|\leq 1 ,\\\\
0 \quad&\mbox{if}&|x|\geq 1,
\end{array}
\right.
\end{eqnarray*}
and
\begin{eqnarray*}
g_n(x)&=&\; \left\{
\begin{array}{cllll}\sqrt{\frac{\alpha_n}{2\pi}} \quad&\mbox{if}&\quad
|x-x_n|\leq{\rm e}^{-2\alpha_n},\\\\-\frac{\log|x-x_n|}{\sqrt{2\pi\alpha_n}}-\sqrt{\frac{\alpha_n}{2\pi}}\quad
&\mbox{if}&\quad {\rm e}^{-2\alpha_n}\leq |x-x_n|\leq {\rm e}^{-\alpha_n},\\\\
0 \quad&\mbox{if}&|x-x_n|\geq {\rm e}^{-\alpha_n}\,.
\end{array}
\right.
\end{eqnarray*}
It follows that, $u_n(x)=2\sqrt{\frac{\alpha_n}{2\pi}} $ for $|x-x_n|\leq {\rm e}^{-2\alpha_n}$. Hence, if
$$
\int_{\R^2}\;\left({\rm e}^{|\frac{u_n(x)}{\lambda}|^2}-1\right)\,dx\leq \kappa,
$$
then
$$
2\pi\int_0^{{\rm e}^{-2\alpha_n}}\,\left({\rm e}^{\frac{2\alpha_n}{\pi\lambda^2}}-1\right)\,r\,dr\leq\kappa,
$$
which implies that
$$
\lambda^2\geq \frac{2\alpha_n}{\pi\log\left(1+\frac{\kappa}{\pi}{\rm e}^{4\alpha_n}\right)}\to \frac{1}{2\pi},
$$
and concludes the proof of the Lemma.
\end{proof}
Let us now go to the proof of Proposition \ref{sumOrlicz2}. To avoid heaviness, we shall restrict ourselves to the example
$h_n(x):=a\,f_{\alpha_n}(x-x_n)+b\,f_{\beta_n}\,(x)$ where  $a, b$ are two real
numbers, $(\alpha_n) \ll (\beta_n)$ are two orthogonal scales and  $(x_n)$ is a core such that
\bq
\label{limit}
-\frac{\log|x_n|}{\alpha_n}\to l\geq 0\,.
\eq
Our purpose is  to show that
$$
\|h_n\|_{\cL}\to \frac{M}{\sqrt{4\pi}}\quad\mbox{as}\quad
n\to\infty,
$$
where $M:=\sup(|a|,|b|)$. Let us  start by proving that
\bq
\label{Inf1}
\liminf_{n\to\infty}\,\|h_n\|_{\cL}\geq\frac{M}{\sqrt{4\pi}}.
\eq
By definition
$$h_n(x) =  b\sqrt{\frac{\beta_n}{2\pi}} + a\,f_{\alpha_n}(x-x_n), \quad
\mbox{if} \quad  |x|\leq {\rm e}^{-\beta_n}\,.$$
Since $0 \leq f_{\alpha_n}(x-x_n) \leq \sqrt{\frac{\alpha_n}{2\pi}}$ and $(\alpha_n) \ll (\beta_n)$, we deduce that for $n$ large enough and $|x|\leq {\rm e}^{-\beta_n}$
$$ |h_n(x)| \geq |b|\sqrt{\frac{\beta_n}{2\pi}}-|a|\sqrt{\frac{\alpha_n}{2\pi}} \,.$$
Therefore, if  $\lambda$ is a positive real such that
\bq
\label{gencond}
\int_{\R^2}\,\left({\rm
e}^{\frac{|h_n(x)|^2}{\lambda^2}}-1\right)\,dx\,\leq\,\kappa\,,
\eq
then
\bq
\label{inf2}
\int_{\{|x|\leq {\rm e}^{-\beta_n}\}}\; \left({\rm
e}^{\frac{\left(|b|\sqrt{\frac{\beta_n}{2\pi}}-|a|\sqrt{\frac{\alpha_n}{2\pi}}\right)^2}{\lambda^2}}-1\right)\,dx\,\leq\,\kappa.
\eq
This implies that
$$
\lambda^2\geq\frac{\left(|b|\sqrt{\frac{\beta_n}{2\pi}}-|a|\sqrt{\frac{\alpha_n}{2\pi}}\right)^2}{2\pi\log\left(1+C{\rm e}^{2\beta_n}\right)}=\frac{b^2}{4\pi}+\circ(1)\,,
$$
and thus
\bq
\label{limitinfb}
\liminf_{n\to\infty}\,\|h_n\|_{\cL}\geq\frac{|b|}{\sqrt{4\pi}}.
\eq
To end the proof of \eqref{Inf1}, it remains to  establish that
\bq
\label{limitinfa}
\liminf_{n\to\infty}\,\|h_n\|_{\cL}\geq\frac{|a|}{\sqrt{4\pi}}.
\eq
For this purpose,  we shall distinguish the case where $ l= \displaystyle\lim_{n\to\infty}\,-\frac{\log|x_n|}{\alpha_n} = 1$  from the one where $l \neq1$. \\

Let us  begin by the most delicate case where $l =1$ and fix
$\delta > 0$. Clearly,  there exists an integer $N$ such that for all $n > N$
\bq
\label{encad}{\rm e}^{- \alpha_n (1+ \delta  )} \leq |x_n| \leq {\rm e}^{- \alpha_n (1- \delta  )}\,.\eq
Now, let us  consider the set
$$E^\delta_n := \{ x\in \R^2; |x -x_n|\leq {\rm e}^{-\alpha_n (1+2\delta)} \,\, \mbox{and  } \quad |x |\geq {\rm e}^{-\beta_n } \}. $$
In view of \eqref{encad}, the reverse triangle inequality implies   that  $ |x | \geq \frac { 1 }{2} {\rm e}^{- \alpha_n (1+ \delta)} $  for all $x \in E^\delta_n$ and $n$ big enough. We deduce that
  \begin{eqnarray*} |h_n(x)| &\geq&  |a|\sqrt{\frac{\alpha_n}{2\pi}}+|b|\frac{\log|x|}{\sqrt{2\pi\beta_n}} \geq  |a|\sqrt{\frac{\alpha_n}{2\pi}}-|b|\frac{\alpha_n (1+ \delta + \circ(1))}{\sqrt{2\pi\beta_n}}\\ & \geq & |a| \sqrt{\frac{\alpha_n}{2\pi}} \big(1-|\circ(1)|\big)\,,\end{eqnarray*}
for all $x \in E^\delta_n$ and $n$ sufficiently large. Consequently,  if   estimate \eqref{gencond} holds, then
\bq
\label{inf1}
\int_{ E^\delta_n }\; \left({\rm
e}^{\frac{\left(|a|\sqrt{\frac{\alpha_n}{2\pi}}\big(1-|\circ(1)|\big)\right)^2}{\lambda^2}}-1\right)\,dx\,\leq\,\kappa,
\eq
which leads to
$$
\lambda^2  \geq \frac{\left(|a|\sqrt{\frac{\alpha_n}{2\pi}}\big(1-|\circ(1)|\big)\right)^2}{\log\left(1+C{\rm e}^{2\alpha_n (1+2\delta)}\right)}=\frac{a^2}{4\pi (1+2\delta)}+\circ(1),
$$
and ensures that
$$
\liminf_{n\to\infty}\,\|h_n\|_{\cL}\geq\frac{1}{\sqrt{1+2\delta}}\,\frac{|a|}{\sqrt{4\pi}}\,.
$$
This achieves the proof of  \eqref{limitinfa} by letting $\delta$ to $0$.\\

The case where $l \neq1$ can be handled in a similar way once we observe   that for n large enough there exists a positive constant $c$ such that $$ |h_n(x)| \geq  |a|\sqrt{\frac{\alpha_n}{2\pi}}+|b|\frac{\log|x|}{\sqrt{2\pi\beta_n}} \geq |a| \sqrt{\frac{\alpha_n}{2\pi}}- c\,|b|\frac{\alpha_n}{\sqrt{2\pi\beta_n}}, $$
for $x \in E_n = \{ x\in \R^2; \, \frac 1 2 \, {\rm e}^{-\alpha_n }  \leq |x -x_n|\leq {\rm e}^{-\alpha_n } \,\, \mbox{and  } \quad |x |\geq {\rm e}^{-\beta_n } \} $. \\

In the general case, we have to replace  \eqref{inf2} and \eqref{inf1} by $\ell$ estimates
of that type. Indeed, assuming that $ \frac{\alpha_n^{(j)}}{ \alpha_n^{(j+1)} }  \to 0 $ when
n goes to infinity for $j = 1, 2, ..., l-1$, we replace \eqref{inf2} and \eqref{inf1}  by the fact that
\bq
\label{inf3}
\int_ {|x -x^{(\ell)}_n|\leq {\rm e}^{-\alpha_n^{(\ell)} }} \; \left({\rm
e}^{\frac{|h_n(x)|^2}{\lambda^2}}-1\right)\,dx\,\leq\,\kappa,
   \eq
and for $j = 1, ..., l-1$
\bq  \label{inf4}
\int_{E^{j,\ell}_n }  \; \left({\rm
e}^{\frac{|h_n(x)|^2}{\lambda^2}}-1\right)\,dx\,\leq\,\kappa,
\eq
in the case where $l=  \displaystyle\lim_{n\to\infty}\,-\frac{\log|x^{(j)}_n|}{\alpha_n} \neq 1$ with
$$ E^{j,\ell}_n = \{ x\in \R^2; \, \frac 1 2 {\rm e}^{-\alpha_n^{(j)} } \leq |x -x^{(j)}_n|\leq {\rm e}^{-\alpha_n^{(j)} } \,\, \mbox{and}\,\,  |x -x^{(j')}_n|\geq {\rm e}^{-\alpha_n^{(j')} }\,\, \mbox{for}\,\, j+1 \leq j'\leq \ell \}\,,  $$
 or
 \bq  \label{inf5}
\int_{E^{j,\ell,\delta}_n }  \; \left({\rm
e}^{\frac{|h_n(x)|^2}{\lambda^2}}-1\right)\,dx\,\leq\,\kappa
\eq
in the case where $l=  \displaystyle\lim_{n\to\infty}\,-\frac{\log|x^{(j)}_n|}{\alpha_n} = 1$ with
$$ E^{j,\ell,\delta}_n = \{ x\in \R^2; \,  |x -x^{(j)}_n|\leq {\rm e}^{- (1+2 \delta)\alpha_n^{(j) } } \,\, \mbox{and}\,\,  |x -x^{(j')}_n|\geq {\rm e}^{-\alpha_n^{(j')} }\,\, \mbox{for}\,\, j+1 \leq j'\leq \ell \}\,. $$

$ $

Our concern now is to prove the second part, that is
\bq
\label{Sup}
\limsup_{n\to\infty}\,\|h_n\|_{\cL}\leq\frac{M}{\sqrt{4\pi}}\,.
\eq
To do so, it is sufficient to show that for any $\eta>0$ small enough and $n$ sufficiently large
\bq
\label{sup1}
\int_{\R^2}\,\left({\rm
e}^{\frac{4\pi-\eta}{M^2}|h_n(x)|^2}-1\right)\,dx\,\leq\,\kappa.
\eq
Actually, we will prove that the left hand side of  \eqref{sup1} goes to zero when
$n$ goes to infinity. For this purpose, write
\begin{eqnarray}
\label{ABC}
\nonumber
\frac{(4\pi-\eta)}{M^2}\,|h_n(x)|^2&=&\frac{4\pi-\eta}{M^2}\,a^2 f_{\alpha_n}(x-x_n)^2+\frac{4\pi-\eta}{M^2}\,b^2f_{\beta_n}(x)^2\\ \nonumber&+&2\frac{4\pi-\eta}{M^2}\,ab\,f_{\alpha_n}(x-x_n)\,f_{\beta_n}(x)\\
&:=&A_n+B_n+C_n.
\end{eqnarray}
The simple observation
\begin{eqnarray*}
 {\rm e}^{x+y+z}-1&=& \left({\rm
e}^{x}-1\right)\left({\rm e}^{y}-1\right)\left({\rm
e}^{z}-1\right)+\left({\rm e}^{x}-1\right)\left({\rm
e}^{y}-1\right)\\ \nonumber&+&\left({\rm e}^{x}-1\right)\left({\rm
e}^{z}-1\right)+\left({\rm e}^{y}-1\right)\left({\rm
e}^{z}-1\right)\\ \nonumber&+&\left({\rm e}^{x}-1\right)+\left({\rm
e}^{y}-1\right)+\left({\rm e}^{z}-1\right),
\end{eqnarray*}
 yields
\begin{eqnarray}
\label{Obs}
\nonumber
\int_{\R^2}\,\left({\rm
e}^{\frac{4\pi-\eta}{M^2}|h_n(x)|^2}-1\right)\,dx\,&=&\,\int\,\left({\rm
e}^{A_n}-1\right)\left({\rm e}^{B_n}-1\right)\left({\rm
e}^{C_n}-1\right)+\int\,\left({\rm
e}^{A_n}-1\right)\left({\rm
e}^{B_n}-1\right)\\&+&\int\,\left({\rm
e}^{A_n}-1\right)\left({\rm
e}^{C_n}-1\right)+\int\,\left({\rm
e}^{B_n}-1\right)\left({\rm
e}^{C_n}-1\right)\\ \nonumber&+&\int\,\left({\rm
e}^{A_n}-1\right)+\int\,\left({\rm
e}^{B_n}-1\right)+\int\,\left({\rm e}^{C_n}-1\right)\,.
\end{eqnarray}
We will demonstrate that each term in the right hand side of \eqref{Obs} tends to zero as $n$ goes to infinity. Let us first observe that, by Trudinger-Moser estimate \eqref{Mos1}, we have for $\varepsilon\geq 0$ small enough
 \bq
 \label{cv1}
 \Big\|{\rm e}^{A_n}-1\Big\|_{L^{1+\varepsilon}}+\Big\|{\rm e}^{B_n}-1\Big\|_{L^{1+\varepsilon}}\to 0\quad\mbox{as}\quad n\to\infty.
 \eq
Concerning the last term in \eqref{Obs}, we infer that for any $\gamma \geq 0$
\bq
\label{lastterm}
\int_{\R^2}\,\left({\rm e}^{\gamma\,f_{\alpha_n}(x-x_n)f_{\beta_n}(x)}-1\right)\,dx\to 0,\quad n\to\infty\,.
\eq
Indeed
\begin{eqnarray}
\int_{\R^2}\,\left({\rm e}^{\gamma\,f_{\alpha_n}(x-x_n)f_{\beta_n}(x)}-1\right)\,dx&=&\int_{|x|\leq{\rm e}^{-\beta_n}}\,\left({\rm e}^{\gamma\,f_{\alpha_n}(x-x_n)f_{\beta_n}(x)}-1\right)\,dx\nonumber\\&+&\int_{{\rm e}^{-\beta_n}\leq |x|\leq 1}\,\left({\rm e}^{\gamma\,f_{\alpha_n}(x-x_n)f_{\beta_n}(x)}-1\right)\,dx\,.
\label{lastterm1}
\end{eqnarray}
The first integral in the right hand side of \eqref{lastterm1} is  dominated by $\pi\,{\rm e}^{-2\beta_n}\,{\rm e}^{\frac{\gamma}{2\pi}\sqrt{\alpha_n\beta_n}}$, which tends to zero as $n$ goes to $\infty$. The second integral can be estimated by
$$
2\pi\int_{{\rm e}^{-\beta_n}}^1\,\left(r^{-\frac{\gamma}{2\pi}\sqrt{\frac{\alpha_n}{\beta_n}}}-1\right)\,r\,dr=2\pi
\left[\frac{1}{2-\frac{\gamma}{2\pi}\sqrt{\frac{\alpha_n}{\beta_n}}}-\frac{1}{2}+\frac{{\rm e}^{-2\beta_n}}{2}-\frac{{\rm e}^{-\beta_n\left(2-\frac{\gamma}{2\pi}\sqrt{\frac{\alpha_n}{\beta_n}}\right)}}{2-\frac{\gamma}{2\pi}\sqrt{\frac{\alpha_n}{\beta_n}}}\right],
$$
which also tends to zero as $n\to\infty$. This gives \eqref{lastterm} as desired.\\

Making use of  H\"older inequality, \eqref{cv1} and \eqref{lastterm}, we get (for $\varepsilon>0$  small enough)
\begin{eqnarray*}
\int\,\left({\rm e}^{A_n}-1\right)\left({\rm
e}^{C_n}-1\right)+\int\,\left({\rm e}^{B_n}-1\right)\left({\rm
e}^{C_n}-1\right)&\leq& \Big\|{\rm
e}^{A_n}-1\Big\|_{L^{1+\varepsilon}}\;\Big\|{\rm
e}^{C_n}-1\Big\|_{L^{1+\frac{1}{\varepsilon}}}\\&+&\Big\|{\rm
e}^{B_n}-1\Big\|_{L^{1+\varepsilon}}\;\Big\|{\rm
e}^{C_n}-1\Big\|_{L^{1+\frac{1}{\varepsilon}}}\to 0\;.
\end{eqnarray*}

Now, we claim that
\bq \label{diff} \int\,\left({\rm e}^{A_n}-1\right)\left({\rm
e}^{B_n}-1\right)\to 0,\;\;n\to\infty\,. \eq

The main difficulty in the proof of \eqref{diff}  comes from the term
$$
\int_{\{\,{\rm e}^{-\beta_n}\leq |x|\leq {\rm e}^{-\alpha_n}\,\}}\;\left({\rm e}^{A_n}-1\right)\left({\rm
e}^{B_n}-1\right)\,\lesssim{\rm e}^{\frac{(4\pi-\eta)a^2\alpha_n}{2\pi M^2}}\,\int_{{\rm e}^{-\beta_n}}^{{\rm e}^{-\alpha_n}}\;{\rm e}^{\frac{(4\pi-\eta)b^2}{2\pi\beta_n M^2}\log^2 r}\,r\,dr,
$$
where we have used the simple fact that $\left(f_{\alpha_n}(x-x_n)\right)^2\leq \frac{\alpha_n}{2\pi}$. The fact that $${\rm e}^{\frac{(4\pi-\eta)a^2\alpha_n}{2\pi M^2}}\,\int_{{\rm e}^{-\beta_n}}^{{\rm e}^{-\alpha_n}}\;{\rm e}^{\frac{(4\pi-\eta)b^2}{2\pi\beta_n M^2}\log^2 r}\,r\,dr \to 0, \quad \mbox{as} \quad n \to \infty$$
is an immediate consequence of the following variant of Lemma 2.8. proved in \cite{BMM}.
\begin{lem}
\label{variantJFA}
Let $(\alpha_n)\ll(\beta_n)$ two orthogonal scales, and $0<p,q<2$ two real numbers. Set
$$
{\mathbf K}_n={\rm e}^{p\alpha_n}\,\int_{{\rm e}^{-\beta_n}}^{{\rm e}^{-\alpha_n}}\;{\rm e}^{q\frac{\log^2 r}{\beta_n}}\,r\,dr\,.
$$
Then ${\mathbf K}_n\to 0$ as $n\to\infty$.
\end{lem}
Finally, since for  $\varepsilon\geq 0$  small enough \eqref{diff}  holds with $A_n$ and $B_n$  replaced by $(1+\varepsilon)A_n$ and $(1+\varepsilon)B_n$,
 the first term in \eqref{Obs} can be treated by the use of H\"older inequality, \eqref{lastterm} and \eqref{diff}.\\

 Consequently, we obtain
$$
\limsup_{n\to\infty}\,\|h_n\|_{\cL}\leq\frac{M}{\sqrt{4\pi}}\,,
$$
which ensures the result. \\

In the general case we replace \eqref{ABC}, by $\ell + \frac{\ell(\ell-1)}{2}$ terms and the rest of the proof is very similar. This completes the proof of Proposition \ref{sumOrlicz2}.


\subsection{Appendix C: Rearrangement of functions and capacity notion}
\label{reacap}
In this appendix we shall sketch in the briefest possible way all useful, known results  on  rearrangement of functions and capacity notion which are used in this article. \\

We start with an overview of the rearrangement of functions. This topic mixes geometry and integration in an essential way. It consists in associating to any measurable function vanishing at infinity, a nonnegative decreasing radially symmetric function. This process minimizes energy, preserves Lebesgue norms, and possesses some other useful and interesting properties that will be given later.\\

To begin with, let us first define the symmetric rearrangement of a measurable set.
\begin{defi}
\label{rearset}
Let $A\subset\R^d$ be a Borel set of finite Lebesgue measure. We define $A^*$, the symmetric rearrangement of $A$, to be the open ball centered at the origin whose volume is that of $A$. Thus,
$$A^*=\{x\,:\, |x|<R\}\quad\mbox{with}\quad\left(|\mathbb{S}^{d-1}|/d\right)R^d=|A|\,,
$$
where $|\mathbb{S}^{d-1}|$ is the surface area of the unit sphere $\mathbb{S}^{d-1}$.

\end{defi}
This definition allows us to define in an obvious way the symmetric-decreasing rearrangement of a characteristic function of a set, namely
$$
\chi^*_{A}:=\chi_{A^*}\,.
$$
To define  the rearrangement of a measurable function $f :\R^d\to \R$, we make use of  the {\it layer cake representation} which is a simple application of Fubini's theorem
$$
|f(x)|=\int_0^\infty\,\chi_{\{|f|>t\}}(x)\,dt\,.
$$
More precisely \begin{defi}
\label{rearfunc}
Let $f :\R^d\to \C$ be a measurable function vanishing at infinity {\rm i.e.}
$$
\forall\;t>0,\qquad\; \Big|\{x\,:\,|f(x)|>t\}\Big|<\infty\,.
$$
We define the symmetric decreasing rearrangement, $f^*$, of $f$ as
\bq
\label{Rear}
f^*(x)=\int_0^\infty\,\chi^*_{\{|f|>t\}}(x)\,dt\,.
\eq
\end{defi}
In the following proposition,  we collect  without proofs  the main features of the rearrangement $f^*$ (for a
complete presentation and more details, we refer the reader to \cite{Kes, Analysis,  Talenti} and the references therein).
\begin{prop}
\label{rearr}
Under the above notation, the following properties hold:
\quad\\
\noindent$\bullet$ The rearrangement $f^*$ of a function $f$ is a nonnegative radially symmetric and nonincreasing function.\\
\noindent$\bullet$ The level sets of $f^*$ are the rearrangements of the level sets of $|f|$, {\rm i.e.},
$$
\{x\,:\,f^*(x)>t\}=\{x\,:\,|f(x)|>t\}^*\,.
$$
In particular, we have
\begin{equation}
\label{meslev}
\Big|\{x\,:\,|f(x)|>t\}\Big|=\Big|\{x\,:\,f^*(x)>t\}\Big|\,,
\end{equation}
where~$|\cdot |$ denotes the Lebesgue measure. \\
\noindent$\bullet$ If $\Phi : \R\to \R$ is non-decreasing, then
\bq
\label{rearang}
\left(\Phi(f)\right)^*=\Phi(f^*).
\eq
\noindent$\bullet$  If $\Phi$ is a Young function, {\rm i.e.}, $\Phi :[0,\infty[\to[0,\infty[$, $\Phi(0)=0$, $\Phi$ is increasing and convex. Then
$$
\int_{\R^d}\,\Phi\left(|\nabla u(x)|\right)\,dx\geq \int_{\R^d}\,\Phi\left(|\nabla u^*(x)|\right)\,dx\,.
$$
\end{prop}
In particular, we derive from Proposition \ref{rearr} that the process of Schwarz symmetrization minimize the energy and preserves Lebesgue and Orlicz norms:
\begin{prop}
\label{norminvar}
Let  $f\in H^1(\R^2)$ then
\beqn
\|\nabla f\|_{L^2}&\geq &\|\nabla f^*\|_{L^2},\\
\|f\|_{L^p}&=&\|f^*\|_{L^p},\\
\|f\|_{\cL}&=&\|f^*\|_{\cL}\,.
\eeqn
\end{prop}

Let us now focus on the notion  of the (electrostatic) capacity giving the elements used along this paper (a detailed exposition on the subject can be found in \cite{Henrot} for example). Before moving on the useful results, let us recall the basic definitions.
\begin{defi}
Let $A\subset\R^d$. The capacity of $A$ is defined by
\bq
\label{capa}
\mbox{\tt Cap}(A)=\inf\Big\{\,\|v\|_{H^1(\R^d)}^2\,;\,v\geq 1 \;\mbox{{\rm a.e.} in a neighborhood of}\, A\,\Big\}\,.
\eq
\end{defi}
In some cases, we need to use the relative capacity defined as follows:
\begin{defi}
\label{relativeCap}
Let $\Omega$ be an open bounded set of $\R^d$. For any subset $A$ of $\Omega$, we define the capacity of $A$ with respect to $\Omega$ by
\bq
\label{relCap}
\mbox{\tt Cap}_\Omega(A)=\inf\Big\{\,\int_\Omega|\nabla v|^2\,;\, v\in H^1_0(\Omega)\,,\,v\geq 1 \;\mbox{{\rm a.e.} in a neighborhood of}\, A\,\Big\}\,.
\eq
\end{defi}
To handle in a practical way the relative capacity, we will resort for $A\subset \Omega$ to the following closed convex subset of $H^1_0(\Omega)$:
$$
\Gamma_A=\Big\{\,v\in H^1_0(\Omega)\,;\;\exists\; v_n\to v,\;v_n\geq 1 \;\mbox{{\rm a.e.} in a neighborhood of}\, A\,\Big\}\,.
$$
This set involves in  the calculation of capacity through the capacitor potential defined as follows:
\begin{defi}
Suppose that $\Gamma_A$ is non empty. The capacitor potential of $A$, $u_A$, is the projection of $0$ on $\Gamma_A$ with respect to $H^1_0(\Omega)$ norm, namely
$$
\int_\Omega|\nabla u_A|^2=\inf\Big\{\,\int_\Omega|\nabla v|^2\;;\;v\in\Gamma_A\Big\}\,.
$$
\end{defi}
The following theorem allows to compute $u_A$ and $\mbox{\tt Cap}_\Omega(A)$.
\begin{thm}
\label{capac}
Let $A\subset \Omega$. Under the above notation, we have\\
\noindent$\bullet$ $\mbox{\tt Cap}_\Omega(A)=\int_\Omega|\nabla u_A|^2$, if $\Gamma_A\neq\emptyset$. \\
\noindent$\bullet$ $u_A$ is harmonic in $\Omega\backslash\bar{A}$.
\end{thm}
As an application, we deduce the capacity of  balls.
\begin{prop} \label{conssym}
Let $0<a<b$. Then
\bq
\label{capBall}
\mbox{\tt Cap}_{B(b)}\left(B(a)\right)=\frac{2\pi}{\log\left(\frac{b}{a}\right)}\,,
\eq
where $B(r)$ denotes the ball of $\R^2$ centered at the origin and of radius $r$.
\end{prop}
\begin{proof}
According to Theorem \ref{capac}, we have
$$
\mbox{\tt Cap}_{B(b)}\left(B(a)\right)=\int_{|x|<b}\,|\nabla u_a(x)|^2\,dx,
$$
where $u_a$ solves the following problem
\begin{eqnarray*}
\left\{
\begin{array}{cllll}\Delta\,u_a=0 \quad&\mbox{if}&\quad
a<|x|<b,\\ u_a=0\quad
&\mbox{if}&\quad |x|=b ,\\
u_a=1 \quad&\mbox{if}&\quad |x|<a.
\end{array}
\right.
\end{eqnarray*}
Straightforward computation yields to
$$
u_a(x)=-\frac{\log\left(\frac{|x|}{b}\right)}{\log\left(\frac{b}{a}\right)}\qquad a<|x|<b\,,
$$
which ensures the result.
\end{proof}
The following lemma shows that the  example by Moser $f_{\alpha}$ is the minimum energy function  which is equal to the value $\sqrt{\frac{\alpha}{2\pi}}$ on the ball $B(0, {\rm e}^{- \alpha})$ and which vanishes outside the unit ball.
\begin{lem}
\label{minEner}
Let $\alpha>0$ and set
$$
K_\alpha:=\Big\{\, u\in H^1_{0,rad}(B(1));\quad u(x)\geq \sqrt{\frac{\alpha}{2\pi}}\quad\mbox{if}\quad |x|\leq {\rm e}^{-\alpha}\,\Big\}\,.
$$
Then
\bq
\label{min1}
\Inf_{u\in K_\alpha}\,\|\nabla u\|_{L^2}^2=1\,.
\eq
\end{lem}
\begin{proof}
Consider the following problem of minimizing
$$
I[u]:=\|\nabla u\|^2_{L^2(B(1))},
$$
among all the functions belonging to the set $ K_\alpha$. Since $K_\alpha$ is a closed convex subset of $H^1_{0,rad}(B(1))$, we get a variational problem with obstacle. It is well known (see for example, L. C. Evans \cite{E} and Kinderlehrer-Stampacchia \cite{KS}) that it has a unique minimizer $u^*$ belonging to $W^{2,\infty}(B(1))$, and which is harmonic in the set $\{\, {\rm e}^{-\alpha}<|x|<1\,\}$. Hence $u^*(x)=a\log|x|$ for ${\rm e}^{-\alpha}<|x|<1$ with some negative constant $a$. Since $u^*({\rm e}^{-\alpha})\geq \sqrt{\frac{\alpha}{2\pi}}$, we get $-a\geq \frac{1}{\sqrt{2\pi\alpha}}$. Therefore
\beqn
\|\nabla u^*\|_{L^2}^2&\geq&a^2\int_{{\rm e}^{-\alpha}<|x|<1}\;\frac{dx}{|x|^2}\\
&\geq&2\pi a^2\alpha\\
&\geq& 1\,.
\eeqn
The fact that $\|\nabla f_\alpha\|_{L^2}^2=1$ concludes the proof.
\end{proof}
We end this section by the following useful result which estimates the minimum energy of a function according to the variation of its values.
\begin{prop} \label{mainextraccore}
Let $B$ a ball of $\R^2$, and $E_1$, $E_2$ two subsets of $B$ such that
$$
0<|E_1|< |E_2|<  |B|\;\;  \quad\mbox{and}\quad |E_1|+|E_2|<|B|\,.
$$
Let $0<a_2<a_1$, and set
$$
K:=\Big\{\;u\in H^1(B);\quad |u|\geq a_1\;\;\;\mbox{on}\;\;\; E_1\quad\mbox{and}\quad |u|\leq a_2\;\;\;\mbox{on}\;\;\; E_2\;\Big\}\,.
$$
Then, for all $u\in K$, we have
\bq
\label{maincore}
\|\nabla u\|_{L^2}^2\geq \frac{4\pi(a_1-a_2)^2}{\log\left(\frac{|B|}{|E_1|}\right)}\,.
\eq
\end{prop}
\begin{proof}
Using Schawrz symmetrization, we can assume that
$$
B=B(R),\quad E_1=B(r_1),\quad\mbox{and}\quad E_2=\{\,r_2<|x|<R\},
$$
where $0<r_1<r_2<R$. Arguing exactly as in the proof of Lemma \ref{minEner}, we obtain
\beqn
\Inf_{u\in K}\,\|\nabla u\|_{L^2}^2&\geq&\left(\frac{a_1-a_2}{\log(\frac{r_1}{r_2})}\right)^2\,\int_{r_1<|x|<r_2}\,\frac{dx}{|x|^2}\\
&\geq&\frac{4\pi(a_1-a_2)^2}{\log(\frac{\pi r_2^2}{\pi r_1^2})}\\
&\geq&\frac{4\pi(a_1-a_2)^2}{\log\left(\frac{|B|}{|E_1|}\right)}\,.
\eeqn

\end{proof}



\bigskip

\bigskip

\noindent{\bf Acknowledgments.} {\it  We are very grateful  to P. G\'erard, L.
Mazet and E. Sandier
 for  interesting
 discussions around the questions dealt with in this paper.} \\


\begin{thebibliography}{10}

\bibitem{AT}
S.~Adachi and K.~Tanaka, {\em Trudinger type inequalities in
$\mathbb R^N$ and their best exponents}, Proceedings of the American Mathematical Society, {\bf
128 (7)} (2000), 2051--2057.

\bibitem{A}
D. R.~Adams, {\em A sharp inequality of J. Moser for higher order derivatives}, Annals of Mathematics, {\bf
128 (2)} (1988), 385--398.

\bibitem{AD} Adimurthi and O. Druet, {\em Blow-up analysis in dimension 2 and a sharp form of
              Trudinger-Moser inequality}, Communications in Partial Differential Equations, {\bf 29} (2004), 295--322.

\bibitem{AT2} Adimurthi and K. Tintarev, {\em On compactness in the Trudinger-Moser inequality}, {\it arXiv:1110.3647}, (2011).

\bibitem{BCD} H. Bahouri, J.-Y. Chemin and R. Danchin, {\em Fourier Analysis and Nonlinear
Partial Differential Equations}, { Grundlehren der mathematischen Wissenschaften, Springer}, (2011).

\bibitem{BCG} H. Bahouri, A. Cohen and G. Koch, {\em A general wavelet-based profile decomposition in the critical embedding of function spaces},   Confluentes Mathematici, {\bf 3} (3)  (2011), 1--25.


\bibitem{BG} H. Bahouri and I. Gallagher,   {\em Weak stability  of the set of global solutions to the Navier-Stokes equations}, {\it arXiv:1109.4043}, (2011).

  \bibitem{BG2}
H.~Bahouri and P.~G\'erard, {\em High frequency approximation of
solutions to critical nonlinear wave equations},  American  Journal of  Math,
 {\bf 121} (1999),  131--175.


\bibitem{BMM}
H.~Bahouri, M.~Majdoub and N.~Masmoudi, {\em On  the lack of compactness in the 2D critical Sobolev embedding}, Journal of Functional Analysis, {\bf 260} (2011),  208--252.


\bibitem{Benameur}  J. Ben Ameur, {\em Description du d\'efaut
de compacit\'e de l'injection de Sobolev sur le groupe de Heisenberg},
   Bulletin de la Soci{\'e}t{\'e}
Math{\'e}\-ma\-ti\-que de  Belgique,  {\bf15} (4)  (2008), 599--624.

\bibitem{BC}  H. Brezis and J. M. Coron, {\em Convergence of solutions of H-Systems or how to blow bubbles},
   Archiv for Rational Mechanics
and Analysis, {\bf 89} (1985), 21--86.

\bibitem{BW} H. Brezis and S. Wainger,  {\em A note on limiting cases of Sobolev embeddings and
              convolution inequalities},  Communications in Partial Differential Equations, {\bf 5} (1980), 773--789.


\bibitem{CC}
L.~Carleson and A.~Chang, {\em on the existence of an extremal
function for an inequality of J. Moser}, Bulletin des Sciences Math\'ematiques, {\bf 110}
(1986), 113--127.


\bibitem{Cianchi} A. Cianchi, {\em A sharp embedding theorem for Orlicz-Sobolev spaces}, Indiana University Mathematics Journal,
 {\bf 45} (1996), 39--65.

\bibitem{EKP} D. E. Edmunds, R. Kerman and L. Pick, {\em Optimal Sobolev imbeddings involving rearrangement-invariant
              quasinorms},  Journal of Functional Analysis, {\bf 170} (2000), 307--355.

\bibitem{E} L. C. Evans, {\em Partial differential equations}, Graduate studies in Mathematics AMS, (1998).

\bibitem{Flu}
M.~Flucher, {\em Extremal functions for the Trudinger-Moser
inequality in 2 dimensions}, Commentarii Mathematici Helvetici, {\bf 67} (1992),
471--479.

\bibitem{FLS}
N. Fusco, P.-L.~Lions and C. Sbordone, {\em Sobolev imbedding theorems in bordeline cases}, Procedings of the  American Mathematical Society, {\bf 124} (1996),
561--565.

\bibitem{Ge2}
P.~G\'erard, {\em Description du d\'efaut de compacit\'e de
l'injection de Sobolev}, ESAIM: Control, Optimisation and Calculus of Variations, {\bf 3} (1998),
213--233 (electronic, URL: http://www.emath.fr/cocv/).

\bibitem{Hans} K. Hansson, {\em Imbedding theorems of Sobolev type in potential theory}, Mathematica Scandinavica, {\bf 45} (1979), 77--102.


\bibitem{HMT} J. A. Hempel,  G. R. Morris and N. S. Trudinger, {\em On the sharpness of a limiting case of the Sobolev imbedding
              theorem }, Bulletin of the Australian Mathematical Society, {\bf 3} (1970),  369--373.


\bibitem{Henrot} A. Henrot and M. Pierre, {\em Variations et optimistation de formes}, {Math\'ematiques et applications}, Springer, {\bf 48}, (2005).

    \bibitem{IB} S. Ibrahim and M. Majdoub,
 {\em Comparaison des ondes lin\'eaires et non lin\'eaires \`a
coefficients variables},   Bulletin de la Soci{\'e}t{\'e}
Math{\'e}\-ma\-ti\-que de  Belgique, {\bf 10} (2003), 299--312.

\bibitem{IMM07}
S.~Ibrahim, M.~Majdoub, and N.~Masmoudi, {\em Double logarithmic inequality with a sharp constant},
{ Proceedings  of the  American Mathematical Society}, (2007) {\bf 135} (1), 87--97.


\bibitem{IMN11TM} S.~Ibrahim, N.~Masmoudi, and K.~Nakanishi, {\em Trudinger-{M}oser inequality on the whole plane with the exact growth
  condition}, arXiv:1110.1712, (2011).


\bibitem{jaffard} S. Jaffard, {\em Analysis of the lack of compactness in the critical Sobolev embeddings},  Journal of Functional Analysis, {\bf 161} (1999),  384--396.

\bibitem{km} C. E. Kenig and F. Merle, {\em Global well-posedness, scattering and blow-up for the energy
critical focusing non-linear wave equation}, Acta Mathematica,  {\bf 201} (2008),
 147--212.

 \bibitem{ker} S. Keraani, {\em On the defect of compactness for the Strichartz estimates  of the Shr\"odinger equation}, Journal of Differential equations, {\bf175} (2001), 353--392.

\bibitem{Kes} S. Kesavan, {\em Symmetrization} \& {\em applications}, {\it Series in Analysis},  {\bf 3}, {\it World Scientific Publishing Co. Pte. Ltd., Hackensack, NJ}, (2006).

\bibitem{KS} D. Kinderlehrer and G. Stampacchia,
{\em An introduction to variational inequalities and their applications},
Academic Press, (1980).

 \bibitem{Kozono} H. Kozono and H. Wadade, {\em Remarks on {G}agliardo-{N}irenberg type inequality with
              critical {S}obolev space and {BMO}},
Mathematische Zeitschrift, {\bf 259} (2008), 935--950.

\bibitem{Analysis} Elliott H. Lieb and Michael Loss, {\em Analysis}, {Graduate Studies in Mathematics, American Mathematical Society},  {\bf 14}, (1997).

\bibitem{La} C. Laurent, {\em On stabilization and control for the critical Klein-Gordon equation on a 3-D compact manifold}, to appear in Journal of Functional Analysis.

\bibitem{Lions1}
P.-L.~Lions, {\em The concentration-compactness principle in the
calculus of variations. The limit case. I.},  Revista Matemática Iberoamericana, {\bf 1} (1985), 145--201.

\bibitem{Lions2}
P.-L.~Lions, {\em The concentration-compactness principle in the
calculus of variations. The locally compact case .I.}, Annales de l'Institut Henri Poincar\'e Analyse Non Lin\'eaire, {\bf 1} (1984), 109--145.

\bibitem{Ma} M. Majdoub,
 {\em Qualitative study of the critical wave equation with a
subcritical perturbation},  Journal of  Mathematics  Analysis
and  Applications, {\bf 301} (2005), 354--365.



\bibitem{MP-PAMS} J. Mal{\'y} and L. Pick, {\em An elementary proof of sharp Sobolev embeddings}, Proceedings of the American Mathematical Society, {\bf 130} (2002), 555--563.



\bibitem{M} J.~Moser, {\em A sharp form of an inequality of N. Trudinger}, Indiana University Mathematics Journal, {\bf 20} (1971), 1077--1092.



 \bibitem{Orlicz-Book} M.-M. Rao and Z.-D. Ren, {\em Applications of Orlicz spaces}, Monographs and Textbooks in Pure and Applied Mathematics, {\bf 250} (2002), Marcel Dekker Inc.



\bibitem{Ruf} B.~Ruf, {\em A sharp Trudinger-Moser type inequality for unbounded domains in }$\mathbb R\sp 2$, Journal of Functional Analysis, {\bf 219} (2005), 340--367.


\bibitem{ST} I. Schindler and K. Tintarev, {\em An abstract version of the concentration compactness principle}, Revista Math Complutense, {\bf 15} (2002), 417--436.

\bibitem{So} S. Solimini, {\em A note on compactness-type properties with respect
to Lorentz norms of bounded subset of a Sobolev space}, Annales de  l'IHP analyse non lin\'eaire, {\bf 12} (1995), 319--337.


\bibitem{St} M. Struwe, {\em A global compactness result for boundary value problems involving limiting
nonlinearities}, Mathematische Zeitschrift, {\bf 187} (1984), 511--517.

\bibitem{Struwe88} M. Struwe, {\em Critical points of embeddings of {$H^{1,n}_0$} into Orlicz spaces}, Annales de l'Institut Henri Poincar\'e Analyse Non Lin\'eaire, {\bf 5} (1988), 425--464.


\bibitem{Struwe10} M. Struwe, {\em Global well-posedness of the Cauchy problem for a super-critical nonlinear wave equation in two space
    dimensions}, Mathematische Annalen,  {\bf 350} (2011), 707--719.

\bibitem{Talenti} G. Talenti, {\em Inequalities in rearrangement invariant function spaces}, {Nonlinear analysis, function spaces and applications}, {\bf 5} (1994), 177--230.

\bibitem{Tao} T. Tao, {\em An inverse theorem for the bilinear $L^2$ Strichartz estimate for the wave equation},
arXiv: 0904-2880, (2009).

\bibitem{Tru}
N.S.~Trudinger, {\em On imbedding into Orlicz spaces and some
applications}, Journal of Applied Mathematics and Mechanics, {\bf 17} (1967),  473--484.

\end{thebibliography}
\end{document}